\documentclass{amsart}
\usepackage{pdfsync}
\usepackage{xspace}
\usepackage[all]{xy}
\usepackage{enumerate}
\usepackage{amssymb}
\usepackage{amsmath}

\usepackage{tikz}
\usetikzlibrary{patterns}

\makeatletter
\DeclareRobustCommand*\cal{\@fontswitch\relax\mathcal}
\makeatother

\newtheorem{definition}{Definition}[section]
\newtheorem{lem}[definition]{Lemma}
\newtheorem{prop}[definition]{Proposition}
\newtheorem{thm}{Theorem}
\newtheorem*{thm*}{Theorem}
\newtheorem{quest}{Question}

\newtheorem{cor}[definition]{Corollary}
\newtheorem{rem}[definition]{Remark}
\newtheorem{exam}[definition]{Example}
\newtheorem{fact}[definition]{Fact}

\newcommand{\rat}{\subseteq_{{rat}}}
\newcommand{\ess}{\subseteq_{{ess}}}

\newcommand{\at}{\mathop{{\rm at}}\nolimits}

\newcommand{\ann}{\mathop{{\rm ann}}\nolimits}
\newcommand{\lann}{\mathop{{\mtx l}}\nolimits}
\newcommand{\rann}{\mathop{{\mtx r}}\nolimits}
\newcommand{\Sl}{\mathop{{S_{left}}}\nolimits}
\newcommand{\Sr}{\mathop{{S_{right}}}\nolimits}
\newcommand{\Sreg}{\mathop{{S_{reg}}}\nolimits}

\newcommand{\mtx}{\mathfrak}

\newcommand{\ifff}{{if and only if\/ }}

\newcommand{\Mod}{\mathop{\rm Mod}}

\newcommand{\RMod}{R\textrm{-Mod}}
\newcommand{\Rmod}{R\textrm{-mod}}

\newcommand{\ModR}{\textrm{Mod-}R}

\newcommand{\Ab}{\mathop{\rm Ab}}

\newcommand{\Hom}{\mathop{\rm Hom}}

\newcommand{\D}{{\rm D}}
\newcommand{\E}{{\rm E}}

\newcommand{\tor}{\mathfrak{s}}

\newcommand{\ulm}{\mathfrak{u}}
\newcommand{\ulmd}{\mathfrak{u}_{div}}
\newcommand{\twinulm}{\mathfrak{v}}

\newcommand{\infulm}{\ulm^\infty}

\newcommand{\ov}{\overline}

\newcommand{\eq}{\,\dot=\,}

\newcommand{\br}{\ov}

\newcommand{\cC}{{\cal C}}
\newcommand{\cD}{\cal D}

\newcommand{\cX}{{\cal X}}

\renewcommand{\phi}{\varphi}

\newcounter{numeq}
\newcounter{num}
\newcounter{one}
\newcounter{two}
\newcounter{three}
\newcounter{four}
\newcounter{five}
\newcounter{six}
\newcounter{seven}
\newcounter{eight}
\newcounter{nine}
\newcounter{ten}
\newcounter{eleven}
\newcounter{twelve}
\newcounter{thirteen}
\newcounter{fourteen}
\newcounter{fivteen}

\newcommand{\thenumm}{{\rm (\alph{num})}}
\newcommand{\thenummeq}{{\rm (\roman{numeq})}}

{\begin{list}{\thenumm}{\usecounter{num}}}%
{\end{list}}

{\begin{list}{\thenumm}{\usecounter{num}
\leftmargin0cm\itemsep0ex\parsep0ex\labelwidth0pt
\itemindent0.5em
\topsep0ex plus0ex minus0ex\parskip0ex}}%
{\end{list}}

{\begin{list}{\thenummeq}{\usecounter{numeq}\setlength{\leftmargin}{2em}
}}%
{\end{list}}

{\begin{list}{\thenummeq}{\usecounter{numeq}
\leftmargin1.1cm\itemindent0cm\itemsep0ex\parsep0ex
\topsep0ex plus0ex minus0ex\parskip0ex}}%
{\end{list}}

\begin{document}
\title{High and low formulas in modules}
\author{Philipp Rothmaler}
\date{\today}
\dedicatory{To Anatole Otto Haru}
\maketitle
\begin{abstract}
A partition of the set of unary pp formulas into four regions is presented, which has a bearing on various structural properties of modules. The machinery developed allows for applications to IF, weakly coherent, nonsingular, and reduced rings, as well as domains, specifically Ore domains. One of the four types of formula are called high. These are used to define Ulm submodules and Ulm length of modules over any associative ring. It is shown that pure injective modules have Ulm length at most 1. As a consequence, pure injective modules over RD domains (in particular, pure injective modules over the first Weyl algebra over a field of characteristic 0)  are shown to decompose into a largest injective and a reduced submodule.
This study serves as preparation for forthcoming work with A.\ Martsinkovsky on injective torsion.
 \end{abstract}

This paper grew out of an attempt to write an introduction to \cite{MR???}, and eventually outgrew those confines by becoming an independent paper on two dichotomies in the lattice of unary pp formulas (over any associative ring $R$ with $1$). 

Call a functor from $\RMod$ to $\Ab$ \texttt{high} if it acts as the forgetful functor on all injectives, and  \texttt{low} if it vanishes on all flats. Applied to unary positive primitive  (henceforth \texttt{pp}) formulas, we obtain high and low formulas with the extra connection that the elementary dual of a high formula (for $\RMod$) is a low formula for $\ModR$, and v.v. (From now on, unless otherwise specified, \texttt{formula} means unary pp formula.) The first dichotomy is this: every (unary pp) formula that is not high is \texttt{bounded} in the sense that it be annihilated, uniformly in \emph{every} module, by a single nonzero scalar (and v.v.). No formula is both, high and bounded. Elementary duality yields a second dichotomy: every formula that is not low is \texttt{cobounded} in the sense that one single nonzero scalar sends  \emph{every} $R$-module into the  pp subgroup defined by that formula. And no formula is both, low and cobounded.

These two dichotomies combined partition the (unary pp) lattices on both sides into four regions, high and low (W), high and cobounded (N), low and bounded (S), and bounded and cobounded (E), see  figure before Rem.\ \ref{empty}.

It is the objective of this work to investigate these regions and provide tools for their application, the first important one being  injective torsion as to be presented in \cite{MR???}. Part of these investigations revolve about 

\begin{quest}\label{quest1}
 What are the rings that have no high-low formulas (i.e., formulas that are both high and low)?
\end{quest}

One quickly establishes that this is a left-right symmetric property, Rem.\,\ref{reformulation}(6) and that, within  domains, it singles out precisely the class of two-sided Ore domains, see Thm.\,\ref{Ore} below. Another two classes of rings over which this question is settled in the affirmative are  uniform RD rings, Cor\,\ref{uniformRD}(2), and the ($\flat\sharp$) rings of \S\ref{four}.

After listing the preliminaries in \S\ref{1}, the foundations of the above are developed in \S\ref{two}. With these at hand one may skip to  (not necessarily commutative) domains.\\

 \noindent \textbf{Corollary \ref{cordich}.} \emph{$R$ is a domain if and only if every unary pp formula  is either high or low  if and only if there is no formula that is both bounded and cobounded.}\\

 See the figure at the outset of \S\ref{dom}, which illustrates this result, namely that domains are the rings over which 
 the E region of the pp lattice  is empty. While the N region contains $x\eq x$ and the S regions contains $x\eq 0$ and thus are never empty, it seems a rather elusive problem to characterize the rings with no high-low formulas, cf.\ Question \ref{quest1} above. In \S \ref{four}  a sufficient condition on a ring $R$ is singled out, the ($\flat\sharp$) condition,  which stipulates the existence of a nonzero flat and absolutely pure module in $\RMod$. In \S \ref{dom} it is shown that  over domains this condition is   also necessary. More concretely, \\
 
 \noindent \textbf{Theorem \ref{Ore}.} \emph{The following are equivalent for any domain $R$.
   \begin{enumerate}[\rm (i)]
   \item $R$ is two-sided Ore. \stepcounter{enumi}\stepcounter{enumi}
\item $R$ has {\rm ($\flat\sharp$)} on one side. 
\item $R$ is two-sided  {\rm ($\flat\sharp$)}.
\item $R$ has no high-low formula.\\
\end{enumerate}}
 
 In this case, canonical examples of ($\flat\sharp$) modules are  $\E(_RR)$ and  $\E(R_R)$ (and this is equivalent as well, cf.\ Thm.\,\ref{Ore}(ii) and (iii)). Denote the property of a ring $R$ that $\E(_RR)$ is flat by $_R$(H$_{reg}$). This property has some bearing on  injective torsion, cf.\ \cite{MR???}. We discuss it briefly in the context of other conditions implying ($\flat\sharp$) in \S\ref{four}, the most prominent of which is (IF): every injective is flat. We re-examine some classical results of Sabbagh, W\"urfel, and Colby  and extend them thus:\\

  \noindent \textbf{Theorem \ref{Colby1}.} \emph{The following are equivalent for any ring $R$.
 \begin{enumerate}[\rm (i)]
\item $_R$(IF), i.e., every injective left $R$-module is flat.\stepcounter{enumi}\stepcounter{enumi}\stepcounter{enumi}
\item Every left pp formula $\phi$ is equivalent in every absolutely pure left $R$-module to a divisibility formula $\mtx A\, |\, \br x$ below $\phi$.
\item Every right pp formula $\psi$ is equivalent in every flat right $R$-module to an annihilation  formula $\br x\, \mtx A \eq \br 0$ above $\psi$. 
\item $R$ is right absolutely pure and right a-coherent.
    \end{enumerate}}

 Here a ring is called right \texttt{a-coherent} if the (left) annihilator of every pp definable right ideal is the same as the annihilator of a finitely generated subideal, Def.\,\ref{a-}. We were not able to settle what had been left open by Colby, whether there is a left IF ring that is not right coherent,
 Question \ref{Colbyquest}.
 Further, it had to be left open also whether a right $\aleph_0$-injective and right a-coherent ring is right absolutely pure (and hence left IF) in general, Question \ref{Kasch} (after Rem.\,\ref{K}). For Warfield rings we have an affirmative answer.\\

 \noindent \textbf{Proposition \ref{Warfield}.}
 \emph{Suppose $R$ is a left Warfield ring.}
 
  \emph{If $R$ is right $\aleph_0$-injective and right a-coherent,  $R$ is right absolutely pure (hence left IF).}\\

In \S\ref{nonsing}, a few tools for future investigation of the dichotomies over nonsingular rings are developed. One of these is\\

 \noindent \textbf{Lemma \ref{Q}.}  \emph{Suppose  $R\subseteq Q$ is a ring extension such that $R_R \subseteq Q_R$ is a rational extension.}
  \emph{Then a formula is low in ${_R\Lambda}$ if and only if it is low in ${_Q\Lambda}$ .}\\

This readily settles Question \ref{quest1} in the following special case.\\

\noindent \textbf{Cor. \ref{commnonsing}.}  \emph{Suppose $R$ is a right nonsingular ring with maximal right quotient ring $Q$ (and $Q_R=\E(R_R)$).}

 \emph{If $_RQ$ is absolutely pure, then $R$ has no high-low formula. In particular, no commutative nonsingular ring has high-low formulas.}\\
 
In \S\ref{red} we turn to reduced rings.\\

 \noindent \textbf{Corollary \ref{abspurered}.} \emph{Let $M$ be a faithful absolutely pure module over a reduced ring. }
 
 \emph{If $M$ has a maximal pp subgroup, then $M$ is  either pp-simple or else possesses incomparable minimal and  maximal  pp subgroups  whose sum is $M$.}\\

 \noindent \textbf{Corollary \ref{abspuredom}.} \emph{ Let $M$ be an absolutely pure module over a domain. }
 
 \emph{ If $M$  has a maximal  pp subgroup, then $M$ is pp-simple and  torsionfree. 
 }\\
 
If the domain is two-sided Ore,  its proof yields the same even for  divisible modules, Cor.\,\ref{divOredom}.  
For the special case of commutative domains, this is 
  \cite[Cor.1.7]{HP}.

 One realization in \S\ref{dom} about domains is that the first Ulm subgroup of an abelian group is the intersection of all its high subgroups, Rem.\,\ref{firstUlm}. We adopt this as a definition  in \S\ref{Ulm} and thus introduce Ulm submodules and Ulm length  for \emph{any} module over \emph{any} ring. Generalizing a well-known fact about pure injective abelian groups, we have\\

\noindent \textbf{Theorem \ref{length1}.} \emph{Every pure injective module has Ulm length $\leq 1$.}\\

\noindent \textbf{Corollary \ref{elemequ}.} \emph{Every module is elementarily equivalent to a module of  Ulm length at most $1$.}\\

\noindent \textbf{Corollary \ref{decRD}.} \emph{Every pure injective  module over an RD domain has a largest injective direct summand whose complement has first Ulm submodule $0$ (i.e., is reduced).}\\
 
Note, this latter consequence of the theorem applies to the first Weil algebra over a field of characteristic 0. The proof of the theorem itself has little to do with modules---it can be proved, mutatis mutandis, for arbitrary $L$-structures, see Thm.\,\ref{generalL}.

 In conclusion of the paper, \S\ref{concl}, a `co-Ulm theory'  is proposed where cobounded formulas play the role of high formulas.

I am grateful to  Alex Martsinkovsky for explaining injective torsion to me, which helped me realize the importance of what I here call low formulas---high formulas are simply the elementary duals of low formulas. The usage of these terms differs slightly from that of Prest and Puninski in \cite{PP}, it does coincide with theirs though for the more specific rings they deal with. 

I would like to thank Martin Ziegler for  inspiring discussions.

\section{Preliminaries}\label{1}{\bf Ring} means associative ring with $1$, usually denoted $R$ and very rarely assumed to be commutative, not even when a \emph{domain}, which here is an arbitrary ring with no zero divisors. $R^0$ is used for the set of nonzero ring elements.
{\bf Modules} are usually (unitary) left $R$-modules over an associative ring $R$ with $1$. If $M$ is a module and $X$ is a set, $M^X$ denotes the direct product of copies of $M$ indexed by $X$, while $M^{(X)}$ is their direct sum.
{\bf Matrices.} We use P.M.\,Cohn's handy matrix notation: $^mW^n$ denotes the set of all $m\times n$ matrices over  (i.e., with entries from) $W$. Superscripts $1$ may be omitted, so that $^mW$ is the set of all column vectors of length $m$ over $W$ and $W^n$ is that of all row vectors of length $n$ over $W$.
{\bf Annihilators.} Given a module $M$ and a subset $U\subseteq M$, we use $\ann_M U$ for its annihilator in $M$. If $U$ is a singleton $u$, we often write $M[u]$ instead. Given a matrix $\mtx A\in {^mR^n}$,   we abbreviate its left annihilator $\ann_{R^m}\{\mtx A\}$ by $\mtx l (\mtx A)$. Similarly, $\rann (\mtx A)$ denotes its right annihilator 
$\ann_{\,{^n\!R}}\{\mtx A\}$. The multiplicative set $\{r\in R\,|\, \lann(r)=0\}$ (of ring elements that are not right zero divisors) is denoted by $\Sl$. Similarly, $\Sr$ denotes the set, $\{r\in R\,|\, \rann(r)=0\}$, of ring elements that are not left zero divisors. This is also a multiplicative set. Finally, the multiplicative set of regular elements, i.e., the set $\Sl\cap\Sr$, is denoted by $\Sreg$.
{\bf Formulas.} Positive primitive (henceforth, pp) formulas are central in the model theory of modules, cf.\ \cite{P1} or \cite[\S 1.1.1 and p.412]{P2}. Given a natural number $n>0$, an $n$-place pp formula for left $R$-modules, can be written as $\exists \br y (\mtx A\br y \eq \mtx B\br x)$, where $\br x  \in {^nX}$ and $\br y \in {^kX}$ for some natural number $k$ (with $X$ a set of variables, $x_i$, usually indexed by natural numbers $i$) and $\mtx A$ and $\mtx B$ are matching matrices over $R$, i.e., $\mtx A\in {^mR^k}$ and $\mtx B\in {^mR^n}$ for some natural number $m>0$.
Using $|$  for divisibility,  this formula may be abbreviated  as $\mtx A|\mtx B\br x$, which is how we usually write it. (Note, this includes the case of quantifier-free pp formulas $\mtx B\br x\eq \br 0$, which are equivalent to $\br 0|\mtx B\br x$).
Similarly, a pp formula for right $R$-modules is one of the form $\exists \br y (\br y \mtx C \eq \br x \mtx D)$, where now $\br x  \in {X^n}$ and $\br y \in {X^k}$, $\mtx C\in {^kR^m}$ and $\mtx D\in {^nR^m}$. I write $\mtx C|^{op}\,\br x \mtx D$ for this formula.
{\bf Purity.} A submodule $M$ of $N$ is pure in $N$ iff $\phi(M)=M\cap \phi(N)$ (where only $\supseteq$ is nontrivial), for all (unary) pp formulas $\phi$.
{\bf Maps.} Positive formulas are preserved by homomorphisms (in any language), so pp formulas of modules are functorial and can be identified with subfunctors of the forgetful functor from modules to abelian groups. 
{\bf Lattices.}
 $\Lambda^n$ denotes the set of all $n$-place pp formulas. We often identify a formula with its equivalence class, i.e., with the class of pp formulas that are equivalent to it in all of $\RMod$. We use $\sim$ to denote this equivalence. Under this identification $\Lambda^n$ becomes a modular lattice with meet $\wedge$ (conjunction) and join $+$ (sum).  In any given module, meet becomes intersection and sum becomes sum (of aditive subgroups). In particular, evaluation of formulas  in a module $M$ maps the lattice $\Lambda^n$ onto a sublattice of the subgroup lattice of the underlying additive group of $M^n$. 
 We write $_R\Lambda^n$ to indicate that pp formulas for left $R$-modules are concerned. The corresponding lattice for right $R$-modules   is denoted by $\Lambda^n_R$. The associated partial order is denoted by $\leq$, so that $\phi\leq\psi$ if and only if  $\phi(M) \subseteq \psi(M)$ in every $R$-module $M$.
The partial order $\leq$ among pp formulas can be described by statements inside the ring as follows, see \cite[Lemma 8.10]{P1} or  \cite[Cor.1.1.16]{P2}.

\begin{fact}[Lemma Presta]\label{LP}
$\mtx A|\mtx B\br x\leq \mtx C|\mtx D\br x$ (in $_R\Lambda^n$) if and only if the  matrix equations 
$\mtx C  X +  Y \mtx B = \mtx D$ and $ Y \mtx A = \mtx C Z$ 
are simultaneously solvable in $R$.
\end{fact}

Obviously, the greatest element of the pp lattices is (the class of)  $\br x\eq \br x$, which is also denoted by $1$, and the smallest  element is (the class of)  $\br x\eq \br 0$,  denoted also  by $0$. We concentrate on $\Lambda^1$ and 
 often  simply write $\Lambda$  instead.
{\bf Pairs of formulas.} By a \emph{pp pair} we mean a pair, $\phi/\psi$, of unary pp formulas such that $\psi\leq\phi$.
We use Ziegler's original notation $\phi/\psi$ for this object, which always presupposes that $\psi\leq\phi$  (in the lattice $\Lambda^1$).  Note that by elementary duality, $\phi/\psi$ is a pp pair \ifff $\D\psi/\D\phi$ is.
{\bf RD rings and formulas.} An RD ring is one over which purity is Relative Divisibility. Examples are $\mathbb{Z}$ and the first Weyl algebra over a field of characteristic $0$ (see \cite[\S2.4.2]{P2} or \cite{PPR}). When $m=n=k=1$ in the above pp formula, it is unary and has the form $r|sx$ with $r, s \in R$. We call this a \texttt{basic RD formula}, since every unary pp formula over an RD ring is (equivalent to) a finite conjunction of such basic RD formulas,  \cite[Prop.2.4.10]{P2}.
Two special cases are of particular interest: when $s=1\in R$, we obtain the \texttt{divisibility formula} $r|x$, and, when $r=0\in R$, the  \texttt{annihilation formula} $sx \eq 0$. To be RD is a two-sided property, and for an RD domain to be Ore is two-sided as well, which is why we simply speak of RD Ore domains, see \cite{PPR} or \cite{P2}.
%
{\bf Multiples and inverses of formulas.}
Let $\phi$ be a unary\footnote{This can be done the same for any arity, but is not needed here.} pp formula $\phi=\phi(x)$ and  $r\in R$ a scalar.
The multiple $r\phi$ of $\phi$ is  the formula $\exists y ( x\eq r y \wedge \phi( y))$. Clearly, $r\phi$ is a pp formula that defines $r\phi(M)$   in every module $M$, i.e., $(r\phi)(M) = r(\phi(M))$.
Given $A\subseteq M\in \RMod$, we use the inverse notation $r^{-1}A$ to denote the subset $\{b\in M\, |\, rb\in A\}$, that is the preimage of $A$ under the map of left multiplication by the scalar $r$. Applying this to $\phi$, we let $r^{-1}\phi$ be the formula $\exists  y( y \eq r x \wedge \phi( y))$, or simply $\phi(r x)$. 
Clearly, $(r^{-1}\phi)(M)= r^{-1}(\phi(M))$.
For right formulas, scalars and their inverses are written on the right.
%
%

\begin{lem}\label{rinverse}
Suppose $M$ is a module over an arbitrary ring $R$  and  $\phi$ is a unary pp formula such that $r \phi(M) = 0$, where  $r\in R$.
 
 \begin{enumerate}[\rm (1)]
  \item $r^{n+1}(r^{-n}\phi(M))=0$ for all $n\geq 0$.
  \item $\phi(M)\subseteq r^{-1} \phi(M)\subseteq r^{-2} \phi(M)\subseteq r^{-3} \phi(M)\subseteq \ldots$.
 \item If $ \phi(M)  = r^{-1} \phi(M)$, then $rM\,\cap\,\phi(M)=0$. 
 \item  If  $r^{-n}\phi(M) = r^{-(n+1)}\phi(M)$,  then  $r^{n+1}M\,\cap\,\phi(M)=0$. 
%
  \end{enumerate}
\end{lem}
\begin{proof} (1) and (2) are obvious from the hypothesis, $r \phi(M) = 0$.

(3). If $r^{-1} \phi(M)\subseteq  \phi(M)$, then for any $a=rb\in  \phi(M)$ we have $b\in \phi(M)$, hence $a=rb=0$.
By hypothesis, $r$ annihilates $\phi$, which easily  yields (5), and that  $r^{n+1}$ annihilates $r^{-n}\phi(M)$ in $M$, which is (3). This trivially implies (2), as every pp subgroup contains $0$. 
Combined with (1), this same argument proves (4).
\end{proof}

{\bf Definable subcategories.}  The result from \cite{purity}  that a class is closed under direct product, direct limit and pure submodule iff it is axiomatized by pp implications  (cf.\ \cite[Thm.3.4.7]{P2}) can be taken as algebraic and model-theoretic definition of definable subcategory at the same time. (Strictly speaking, a definable subcategory is the full subcategory on such a class.) We use the notation $\langle\cX\rangle$ for the definable subcategory generated by a class $\cX$.
{\bf Elementary duality.} We assume familiarity with elementary duality as presented in \cite[\S 1.3]{P2}. Prest introduced it on the level of formulas, and 
Herzog extended it to axiomatizable classes of modules, in particular to closed sets of the Ziegler spectrum and thus to definable subcategories, where the dual of a definable subcategory $\cD\subseteq \RMod$ axiomatized by the closing of pairs $\phi/\psi$ is defined to be the definable subcategory $\D\cD\subseteq \ModR$ axiomatized by the closing of the corresponding dual pairs $\D\psi/\D\phi$, \cite{Her} (cf.\  \cite[3.2.12,  3.4.16, 5.4.1]{P2}). Note that $\D^2=1$, both on the level of formulas and on the level of definable subcategories.  
Given a left pp formula $\phi = (\mtx A\, |\, \mtx B\,\br x)$ (a shorthand, recall, for $\exists \br y (\mtx A\br y \eq \mtx B\br x)$), its dual, $\D\phi$, is by definition the right pp formula  $\exists \br z(\br x \eq \br z\,  \mtx B \wedge \br z\, \mtx A\eq \br 0)$. The following instance of this will be made use of in \S\ref{a-coh}.

\begin{lem}\label{dualmult} Let $\psi\in {_R\Lambda^1}$ and $s\in R$.

 $\D(s\psi)\, = (\D\psi)s^{-1}$ and $(\D\psi)s\leq \D(s^{-1}\psi)$. We have $(\D\psi)s\sim_M \D(s^{-1}\psi)$ in every right module $M$ with $M[s]=0$.
 \end{lem}
 \begin{proof} Let $\psi$ be $\mtx A\, | \, \mtx b\, x$ with $\mtx A \in {^nR^m}$ and $\mtx b\in {^nR^1}$. Then $\D\psi = \exists \br z(x \eq \br z\,  \mtx b \wedge \br z\, \mtx A\eq \br 0)$ (with $\br z$ a row vector of $n$ variables) and $(\D\psi)s^{-1}=\exists \br z(xs \eq \br z\,  \mtx b \wedge \br z\, \mtx A\eq \br 0)$.
 
 For the first statement, consider $s\psi$, which is the formula $\exists z\exists \br y (x\eq sz \wedge \mtx A\,\br y \eq \mtx b\, z)$). To write this in our standard form, consider the matrix $\mtx C\in {^{n+1}R^{m+1}}$ whose first row (upper block) is $(s, 0, \ldots, 0)$ and whose lower block is $(-\mtx b, \mtx A)$. Consider  $\mtx d\in {^{n+1}R^1}$,  the transpose of $(1, 0, \ldots, 0)$. Obviously, $\mtx C\, |\, \mtx d\,x$ is equivalent to $s\psi$ as spelled out above. Hence $\D(s\psi) = \exists \br z'(x \eq \br z'\,  \mtx d \wedge \br z'\, \mtx C \eq \br 0)$. Here $\br z'$ is the row vector  $(z, \br z)$, with $\br z$ is as above in $(\D\psi)s^{-1}$. Then clearly $x \eq \br z'\,  \mtx d \sim z\eq x$, while $\br z'\, \mtx C \eq \br 0 \sim zs \eq \br z\,\mtx b\, \wedge \br z\, \mtx A\eq \br 0$. Thus $\D(s\psi) \sim \exists \br z(xs\eq \br z\,\mtx b\, \wedge \br z\, \mtx A\eq \br 0)$, which is $(\D\psi)s^{-1}$.

 For the second statement, let  $\phi =  s^{-1}\psi = (\mtx A\, | \, \mtx b\, sx) = (\mtx A\, | \, \mtx bs\,x)$. Then $\D\phi= \exists \br z(x \eq \br z\,  \mtx bs \wedge \br z\, \mtx A\eq \br 0)$. Clearly, if $x$ satisfies $\D\psi$, then $xs$ satisfies $\D\phi$, which proves the second statement. 
 
 For the third, let $xs\in M$ satisfy $\D\phi$. If $M[s]=0$, then clearly $x$ satisfies $\D\psi$.
\end{proof}

{\bf Character modules.} $M^+ := \Hom(M, \mathbb{Q/Z})$ can be used to generate the dual definable subcategory, and we extend this to arbitrary classes $\cC$ by letting $\cC^+$ be  $\{ C^+ \,|\, C\in \cC\}$. Given a definable subcategory $\cD$, its dual  $\D\cD$ is then the definable subcategory generated by $\cD^+$. Further, a character $f\in M^+$ annihilates a pp subgroup $\phi(M)$ if and only if $f\in \D\phi(M^+)$, \cite[Lemma 1.3.12]{P2}.
{\bf Flat modules and absolutely pure modules.} We use  $_R\flat$ (resp., $\flat_R$) for the class of flat left (resp., right) $R$-modules, and   $_R\sharp$  (resp., $\sharp_R$) for the class of absolutely pure left (resp., right) $R$-modules. 
The following three results involving pp formulas are essential to our investigation, Zimmermann's characterization of flatness, \cite{Zim} (cf.\ \cite[Thm.2.3.9]{P2}), the description of absolutely purity given in \cite{PRZ2}  (cf.\ \cite[Prop.2.3.3]{P2}), and Herzog's theorem stating that the definable subcategories generated by the flats on one side and  the absolutely pure modules on the other are mutually dual, \cite{Her} (cf.\  \cite[Thm.5.4.1 and Prop.3.4.26]{P2}).  The last two are easy consequences.

.  

\begin{fact}\label{fact} Let $M\in \RMod$.
\begin{enumerate}[\rm (1)]
 \item $M$ is flat iff $\phi(M)=\phi(_RR)M$, for every unary pp formula $\phi$.
 \item $M$ is absolutely pure  iff $\phi(M)=\ann_M\D\phi(_RR)$, for every unary pp formula $\phi$.
 \item $\D\langle {_R\flat} \rangle = \langle \sharp_R  \rangle$ and $\D\langle {_R\sharp} \rangle = \langle \flat_R  \rangle$.
  \item If a pp pair $\phi/\psi$ opens up in a flat module, it opens up in $_RR$.
 \item If a pp pair $\phi/\psi$ opens up in an absolutely pure module, the dual pair $\D\psi/\D\phi$ opens up in $R_R$.
\end{enumerate}
 \end{fact}

\begin{fact} \label{Wur} Let $M$ be a left $R$-module.
 \begin{enumerate}[\rm (1)] 
 \item  {\rm (Lambek)} $M$ is flat if and only if $M^+$ is absolutely pure  (hence injective).
  \item {\rm (Würfel I)} \cite[Krit.1.5]{Wu}. If $M^+$ is flat,  $M$ is  absolutely pure.
  \item {\rm (Würfel II)} \cite[Satz 1.6]{Wu}.  The converse of (2)  is true for every $M$ if and only if $R$ is left coherent. Cf.\ \cite[Thms.4.3, 4.4]{PRZ}.
\end{enumerate}
\end{fact}

{\bf Divisible modules.} 
%
Recall, a module $M$ is \texttt{divisible} (in the `correct' sense) if, for all $r\in R$, an element $a\in M$ is divisible by $r$ in $M$ whenever it is annihilated by all left annihilators of $r$ in $R$, cf.\  \cite[Def.3.16]{LMR}.
 In other words, the formula $\phi = r|x$ is satisfied by all elements of $M$ that are annihilated by all ring elements in $\D\phi(R_R)$, since $\D\phi = (xr \eq 0)$. This statement is an instance of Fact \ref{fact}(2) above for this particular choice of formula $\phi$, which shows the following well-known 
 
 \begin{fact}\label{fact2} 
 Every absolutely pure module is divisible.
 \end{fact}
 
(For the special case of injectives, see for instance  \cite[Cor.3.17$^\prime$]{LMR}.)
{\bf Variants of injectivity.} By Baer's criterion, a module $M$ is \texttt{injective} iff every map from a left ideal $I$ to $M$ can be extended to (factors through) $_RR$.  Following \cite{ES}, $M$ is  \texttt{$\kappa$-injective} if this is true for every left ideal $I$ that has a set of generators  of cardinality  $<\kappa$. Prominent examples are  \texttt{$2$-injective} modules, which are better known as \texttt{principally injective} or \texttt{$P$-injective} or just \texttt{divisible} modules as above (see also  \cite{NY}.) Another important case is that of  \texttt{$\aleph_0$-injective} modules. We have injective$\implies$absolutely pure$\implies$$\aleph_0$-injective$\implies$divisible. If the ring is left coherent \cite[Prop.3.23]{ES} or an RD ring \cite[Rem.\,2.17]{PPR}, the second arrow can be reversed. (It seems still open for which rings this is true on one side, cf.\ \cite[p.268]{ES}.) The first arrow can be reserved  if and only if the first two arrows can be reversed, which is the case precisely when the ring is left noetherian, \cite[Prop.3.24]{ES}.

All other notation and terminology can be found in the cited texts.

\section{Two dichotomies on unary pp formulas}\label{two}
We start with an old lemma I originally proved  by bare hands (that is, in the style of Robinson, in its Eklof--Sabbagh shade). With Fact \ref{fact} at hand, I can now give a simple conceptual proof.  (The result also appears in   \cite[Lemma 2.4]{HP}.) 

\begin{lem}\label{THELemma}\cite[Part II, Lemma 2]{RegTyp}.
 If $\phi(N)\not=N$ for some absolutely pure left $R$-module $N$, then there is $r\in R^0$ such that $r\phi(M)=0$ for all $M\in \RMod$.
\end{lem}
\begin{proof}
The above description of  $\phi(N)$ as an annihilator implies that, if $\phi(N)\not=N$, the left ideal $\D\phi(R_R)$ cannot be zero. Pick a nonzero element $r$ therein and  let  $M$ be any left $R$-module. In its injective hull, $E$,  again $\phi(E)=\ann_E\D\phi(R_R)$. Thus $r$ annihilates $\phi(E)$ and therefore also $\phi(M)\subseteq \phi(E)$.
\end{proof}

This  lemma states a dichotomy. To
reformulate it in a catchier way, I introduce some terminology.

\begin{definition} Let $\phi$ be a unary (left) pp formula. 

\begin{enumerate}[\rm(1)]
\item  Call  $\phi$ \texttt{high}\footnote{Prest and Puninski \cite[after Fact 2.3]{PP}, or \cite[intro to \S8.2.1]{P2}, call `low' and `high' what I call `bounded' and `cobounded,' respectively. The corresponding terms coincide over two-sided Ore domains.} if $\phi(N)=N$ for every absolutely pure (equivalently, for every injective) module $N$, in which case we say, $\phi$ \texttt{covanishes} on $N$. 

\item $\phi$ is called \texttt{bounded} if there is $r\in R^0$ such that $r\phi(M)=0$ for all $M\in \RMod$, i.e., such that  $\phi\leq rx\eq 0$.
 \item $\phi$ is said to be \texttt{low} if it vanishes on flats (equivalently, {\rm Fact \ref{fact}(1)}, if it vanishes on $_RR$). 
 
\item $\phi$ is  called \texttt{cobounded} if there is $s\in R^0$ such that $sM\subseteq\phi(M)$, i.e.,  $s(M/\phi(M))=0$, for all $M\in \RMod$,\footnote{Beware, $M/\phi(M)$ may not be an  $R$-module.} i.e., such that $s|x\leq \phi$.

\item For short, call a formula \texttt{high-low} if it is both high and low.
\end{enumerate}
 Similarly for right formulas, and we extend this terminology to pp subgroups.
\end{definition}

\begin{rem}\label{fundrem}
\begin{enumerate}[\rm (1)]
\item By {\rm Fact \ref{fact}(1)}, $_RR$ is a test module for lowness. 
\item $\E(_RR)$ is a test module for highness. For if\/ $1\in R$ satisfies $\phi$ in some module $A\supseteq {_RR}$, then, as $1$ in $R$ can be sent to any element $b$ in any module  $B$, and, if $B$ is injective, this map can be factored through the inclusion $_RR\subseteq A$, one obtains $b\in\phi(B)$. Moreover, $\phi$ is high iff\/ $1\in \phi(E(_RR))$.
\item If $r\in R$ is not a right zero divisor, the divisibility formula $r|x$ is high, by Fact \ref{fact2}. For  the converse, see {\rm Cor.\,\ref{highdiv}(2)}.

\item The annihilation formula $rx\eq 0$  is low if and only if $r$ is not a left zero divisor.  Otherwise it is (bounded and) cobounded. More specifically, if\/ $s\in R^0$ and $rs=0$, then $rx\eq 0$ is cobounded (implied) by the formula $s|x$.
\item As all modules over a (von Neumann) regular ring are both flat and absolutely pure, every high pp formula is equivalent to $x\eq x$ and every low formula is equivalent to $x\eq 0$. 
All other formulas are  bounded \emph{and} cobounded. By \cite[Prop.2.32]{MR}, low formulas are equivalent to $x\eq 0$ over any left  absolutely pure ring, see also {\rm Prop.\,\ref{u=1}}. By duality, all high formulas are equivalent to $x\eq x$  over any right  absolutely pure ring. Thus two-sided absolutely pure rings show the same behavior as  regular rings as far as lowness and highness are concerned.
\item More concretely, if $0, 1 \not=e=e^2\in R$, the bounded formula $ex\eq 0$ is equivalent to the cobounded formula $(1-e)|x$ and the cobounded formula $e|x$  is equivalent to the bounded formula $(1-e)x\eq 0$; further, the sum of those two bounded formulas is equivalent to the (high) formula $x\eq x$, while the intersection of the other two, cobounded, formulas is equivalent to $x\eq 0$. \\
Note, over regular rings all unary pp formulas are of this shape (up to $\sim$). 
\end{enumerate}
\end{rem}


\begin{lem}[Duality Lemma]\label{DualityLemma}
\begin{enumerate}[\rm (1)]
\item $\phi$ is high iff\/ $\D\phi$ is low (and vice versa).
\item  $\phi$ is bounded iff\/ $\D\phi$ is cobounded  (and vice versa).
\end{enumerate}
 \end{lem}
 \begin{proof} 
 $\phi$ is high iff the pair $x\eq x/\phi$ is shut  in all absolutely pure (or all injective) modules, and 
 $\phi$ is low iff the pair  $\phi(x)/ x\eq 0$ is shut  in all flat modules (or simply in the corresponding regular module). The first assertion now follows from Fact \ref{fact}(3).

A formula $\phi$ is bounded iff it is below some nontrivial  annihilation formula, i.e., iff
 $\phi(x)\,\leq\, rx\,\eq\,0$ with $r\in R^0$. A formula $\phi$ is cobounded iff it is above some nontrivial divisibility formula, i.e., 
iff $s|x \leq\phi(x)$ with $s\in R^0$.
 By  elementary duality, the former is equivalent to $r|^{op}x \leq \D\phi(x)$, i.e., to $\D\phi(x)$ being cobounded. 
\end{proof}

 
\begin{prop}
\begin{enumerate}[\rm (1)]
 \item {\rm (Dichotomy I)}  Every unary pp formula is either high or bounded (and not both). 
  \item {\rm (Dichotomy II)}  Every unary pp formula is either low or cobounded (and not both). 
   \end{enumerate}
\end{prop}
\begin{proof} Lemma \ref{THELemma} says that a unary formula is high or bounded. That it cannot be both follows from
 an easy argument involving the injective hull of the ring.  Now apply the Duality Lemma  to obtain Dichotomy II.
 \end{proof}

\begin{rem}\label{reformulation} Some reformulations are worth keeping in mind.
\begin{enumerate}[\rm (1)]
\item $\phi$ is low iff\/ $\phi(_RR)=0$. 
 \item $\phi$ is high iff\/ $\D\phi(R_R)=0$. 
\item {\rm (Dichotomy I)} \hspace{1mm}
$\D\phi(R_R)\not=0$ iff $\phi$ is bounded. 
\item {\rm (Dichotomy II)} \hspace{1mm}
$\phi(_RR)\not=0$ iff $\phi$ is cobounded.
\item $\phi$ is high-low iff\/ $\phi$ and\/ $\D\phi$ are high iff\/ $\phi$ and\/ $\D\phi$ are low.
\item In particular, the existence of a high-low formula is a left-right symmetric property of the ring.
\item  $\phi$ is bounded and cobounded iff\/ $\phi$ and\/ $\D\phi$ are bounded iff\/ neither $\phi$ nor\/ $\D\phi$ are low.
\item In particular, the existence of formulas that are bounded and cobounded is a left-right symmetric property of the ring.
\end{enumerate}
 \end{rem}

 Martin Ziegler suggested to  find  a syntactic, or rather ring-theoretic criterion for boundedness. Indeed, Lemma Presta yields one very easily.  
\begin{lem}\label{Ziegler} Let $\mtx A\in {^nR\/^m}$ and $\mtx  b\in {^nR}$.

 The unary pp formula $\mtx A\,|\,\mtx b x$ is bounded if and only if 
 there is $\mtx t \in {R^n}$ such that $\mtx  t \, \mtx A = \bar 0$ and $\mtx t\, \mtx  b \not= 0$.
\end{lem}
\begin{proof} If there is such  $\mtx  t$, then clearly $r := \mtx  t \, \mtx  b\not= 0$ annihilates the (subgroup defined by the)  formula $\mtx  A\,|\,\mtx  b x$ everywhere.

For the nontrivial direction, assume this formula is bounded, i.e., $\mtx  A\,|\,\mtx  b x \leq rx\eq 0$ for some nonzero $r\in R$. Write this implication as $\mtx  A\,|\,\mtx  b x \leq 0\,|\, rx$ in order to apply Lemma Presta, Fact \ref{LP}, as follows (with $\mtx  C=0$ and $\mtx  D=r$). There are matching matrices $X,  Y$ such that   $X\, \mtx  A = \bar 0 \in R^m$ and $r= X\, \mtx   b$. As $X$ must be in ${R^n}$, we write it as a vector, $\mtx  t $, and the assertion follows. 
\end{proof}

 We put this into context by pinning down similar criteria in the ring for the other three properties in question.

\begin{prop} Let $\mtx  A\in {^n\!R\/^m}$ and $\mtx  b\in {^n\!R}$. 

\begin{enumerate}[\rm (1)]
 \item $\mtx  A\,|\,\mtx  b x$ is bounded if and only if\/ $\lann(\mtx  A)\not\subseteq\, \lann(\mtx   b)$. 
 \item  $\mtx  A\,|\,\mtx  b x$ is high if and only if\/ $\lann(\mtx  A)\subseteq\, \lann(\mtx  b)$.
\item  $\mtx  A\,|\,\mtx  b x$ is low if and only if\/   $\mtx  bR \, \cap\,   \mtx  A\, {^m\!R}\, = \{\bar 0\}$ 
and\/ $\rann(\mtx   b) = \{0\}$.
  \item  $\mtx  A\,|\,\mtx  b x$ is cobounded if and only if\/  $\mtx  bR^0 \, \cap\,   \mtx  A\, {^m\!R}\,\not=\emptyset$.
\end{enumerate}
\end{prop}
\begin{proof}
 (1) and (2)  follows from the lemma and Dichotomy I.

Let $\phi$ be the unary formula $\mtx  A\,|\,\mtx  b x$. It is easy to see that if one of the right hand conditions is violated, $\phi$ is not low. Conversely, if $\phi$ is not low, there is $0\not= r\in \phi(_RR)$, hence $\mtx  br \in \mtx A\, {^m\!R}$. If now $\mtx  br=\br 0$, then $\rann(\mtx   b) \not= \{0\}$. If not, the first condition is violated. This proves (3). By the other dichotomy, the given formula is cobounded  iff it satisfies the negation of the condition in (3). But whether it is $\mtx  bR \, \cap\,   \mtx  A\, {^m\!R}\, = \{\bar 0\}$ or\/ $\rann(\mtx   b) = \{0\}$ that is violated, we obtain $r\in R^0$ with  $ \mtx  b r \in \mtx  A\, {^m\!R}$. Conversely, if there is  $r\in R^0$ with $ \mtx  b r \in \mtx  A\, {^m\!R}$, then depending on whether $ \mtx  b r$ is $\bar 0$ or not, the latter or the former half of the condition in (3) is violated.
 This concludes the proof of (4).
\end{proof}

Note that, just as for Lemma Presta in general, there is a ring-theoretic formula defining the situation in question. E.g., a formula $\mtx  A\,|\,\mtx  b x$ is bounded if and only if (the entries of) $\mtx  A$ and $\mtx  b$ satisfy the ring-theoretic formula   
$\exists \bar z (\bar z X \eq \bar 0 \wedge \neg\, \bar z \bar y \eq 0)$, etc. 

This shows that the existence of a unary pp formula having any combination of theses properties is a property of the complete theory of the ring, not only of the ring itself (in other words, elementarily equivalent rings possess the same combinations).

\begin{rem}
 One may prove {\rm (2)} and  {\rm (4)} above directly without appeal to the dichotomies, and instead derive the dichotomies from those descriptions. 

{\footnotesize For instance, by (1) above, if a formula $\phi$ of the form $\mtx  A\,|\,\mtx  b x$ is not bounded, then in any module $M$, for every $a\in M$ and every $\mtx  t \in {R^n}$, we have $\mtx   t \mtx  A = \bar 0$ implies $\mtx  t\, \mtx  b \, a= 0$, which by the Eklof-Sabbagh Consistency Lemma  means that the sentence $\phi(a)$ is consistent with $M$, i.e., there is $N\supseteq M$ in which it is true. If $M$ now is absolutely pure, this means that $\phi(a)$ is true in $M$. As $a\in M$ was arbitrary, $\phi(M) = M$, hence $\phi$ is high.}
\end{rem}

Next, the proposition is specialized to basic RD formulas, which allows for a number of conclusions about them and their role in the dichotomies.
Note that $0\,|\,sx\,\sim\, sx\eq 0$ and $r\,|\,0\,x\,\sim\,x\eq x$ (both  quantifier-free). The following are easily   verified.

\begin{rem}\label{T}
Let $r, s\in R$ and  $T\subseteq R$. 
\begin{enumerate}[\rm (1)]
\item There is $t\in T$ with $t|x\leq r|sx$ if and only if\/  $sT\cap rR\not= \emptyset$. 
\item If $sR^0\,\cap\, rR = \emptyset$ then $sR \cap rR = 0$.
\end{enumerate}
\end{rem}

The first part is used repeatedly (and tacitly) with $T=R^0$ below. Later it will be applied with $T=\Sl$.


\begin{cor}\label{highdiv} Let $r, s \in R$.
\begin{enumerate}[\rm (1)]
\item $r|sx$ is bounded iff\/ $\lann(r)\not\subseteq\, \lann(s)$.
\item  $r|sx$ is high iff\/ $\lann(r)\subseteq\, \lann(s)$.
\item $r|sx$ is low iff\/ $sR \cap rR = 0 = \rann(s)$.
\item $r\,|\,s x$ is cobounded iff\/  $sR^0\,\cap\, rR\not= \emptyset$.\\

\noindent
{\rm Special cases for $s=1$:} 
\item  $r|x$ is bounded iff\/ $r$ is a right zero divisor 

(iff\/ there is some $t\in R^0
$ with $tr=0$ and hence $r|x\leq tx\eq 0$).
\item  $r|x$ is high iff\/ $r\in \Sl$ 

 (iff\/ $r$ is not a right zero divisor iff\/ $r|sx$ is high for every $s\in R$.)
  \item $r|x$ is low  iff\/ $r=0$ 
  
  (iff\/ $r|x\,\sim\, x\eq 0$).
  \item $r|x$ is cobounded iff\/ $r\not=0$.\\

\noindent  
{\rm Special cases for $r=0$.} 
\item $sx\eq 0$  is bounded iff\/ $s\not=0$.
  \item $sx\eq 0$ is high iff\/ $s=0$ 
  
  (iff\/ $sx\eq 0\,\sim\, x\eq x$).
  \item $sx\eq 0$ is low iff\/ $s\in \Sr$ 
  
  (iff\/ $s$ is not a left zero divisor). 
  \item $sx\eq 0$  is cobounded iff\/ $s$ is a left zero divisor 
  
  (iff\/ there is some $t\in R^0$ with $st=0$ and hence $t|x\leq sx\eq 0$).\\ 
  
  \noindent
{{\rm Special cases for $R$, a domain: }
\item   $r|sx$ is bounded iff\/ $r= 0$ and $s\not= 0$ 

(iff\/ $r|sx\,\sim\, sx\eq 0$ with $s\not= 0$).
  \item $r|sx$ is high iff\/ $r\not= 0$ or\/ $s=0$. 
  \item  $r|sx$ is low iff\/ $s\not= 0$ and\/ $sR \cap rR  = 0$.
\item    $r|sx$ is cobounded iff\/ $s= 0$  or  $sR \cap rR  \not= 0$.}
\end{enumerate}
\end{cor}

\begin{cor}\label{RDhi-lo} Let $R$ be a domain.
\begin{enumerate}[\rm (1)]
\item Every bounded formula $r|sx$ is low and quantifierfree.
\item Every cobounded formula $r|sx$ is high.
\item $r|sx$ is high-low if and only if $r, s \not= 0$ and  $rR \cap sR  = 0$.
 \item No formula $r|sx$ is high-low if and only of $R$ is right Ore. 
\end{enumerate}
\end{cor}
\begin{proof} (1): (13)$\implies$(15). (2): (16)$\implies$(14). (3):  (14) \& (15).
(3) yields (4).
\end{proof}

\begin{cor}\label{RD}
 No RD Ore domain has high-low formulas.
\end{cor}
\begin{proof}
 A defining property of RD rings is that every pp formula be equivalent to a (finite) conjunction of basic RD formulas. A conjunction of (any) formulas is high if and only if every conjunct is. If $R_R$ is uniform, clearly the same holds for low (left) formulas. Consequently, over an RD ring $R$ with $R_R$ uniform, if there is a high-low formula, there is a basic RD one (on the left). Now apply (4) above.
\end{proof}

Thm.\,\ref{Ore} below extends this to all (two-sided) Ore domains. 
 Getting back to more general rings we have the following under some sort of Ore condition.

\begin{cor}\label{uniformRD}
\begin{enumerate}[\rm (1)]
\item 
Over a right uniform ring (i.e., a ring $R$ with $R_R$  uniform), 
 no basic (left) RD formula is high-low.
 \item Over a one-sided uniform RD ring, there are no high-low formulas.
 \end{enumerate}
\end{cor}
\begin{proof}
 (1). Consider an RD formula, $r|sx$, that is high. If $s=0$, clearly, this formula is equivalent to $x\eq x$. By  Cor.\,\ref{highdiv}(10),
he same holds true if $r=0$. So we may assume, $r, s\in R^0$. If $r|sx$ is cobounded, we are done (referring to Dichotomy II). If not, by Cor.\,\ref{highdiv}(4)  and Rem.\,\ref{T}(2), $sR\cap rR=0$, contradicting right uniformity.

(2) follows from (1) by the argument at the end of proof of Cor.\,\ref{RD}, provided, $R$ is right uniform. If it is left uniform, it remains to invoke left-right symmetry  (of being RD and) of the existence of high-low formulas, cf.\,Rem.\,\ref{hilosym} below.
 \end{proof}

\begin{lem}\label{fract}
Suppose $S$ is a right denominator set in a ring $R$ and $Q$  is the corresponding right ring of fractions, $RS^{-1}$. Let $\phi$ be a (left) unary pp formula.
 \begin{enumerate}[\rm (1)]
 \item If $\phi$ is low, then $\phi(_RQ)=0$, hence $\phi$ is low also in $Q-\Mod$. 
 \item If $S$ is regular (i.e., does not contain zero divisors), the converse holds.
 \end{enumerate}
 \end{lem}
 \begin{proof}
(1). Write $\phi$ as $\mtx A | \mtx b x$. If $0\not= d\in \phi(_RQ)$, let $\bar d$  witness the truth of $\phi(d)$ in $_RQ$, i.e.,  $\mtx A \br d \eq \mtx b d$. Multiply on the right by a common denominator $s\in S$ of $d$ and all the entries of $\br d$ to get the truth of $\phi(ds)$ in $_RR$. If $d=rs^{-1}\not= 0$, then $r=ds\not= 0$ as well, as desired.

(2) If $S$ is regular, $R\subseteq Q$, so, as $\phi$ is existential, $\phi(_RR)\not= 0$ implies $\phi(_RQ)\not= 0$.
\end{proof}

This will be generalized to  (right) rational  ring extensions  in Lemma \ref{Q} below.

\begin{rem}\label{10.17}
The proof of {\rm (1)} above also shows that $_RQ$ is flat (and thus solves  \cite[Ex.10.17]{LMR}):  for $d, r$ and $s$ as chosen there, we have $d= r s^{-1} \in \phi(_RR)Q$. Since $\phi$ and $d\in \phi(_RQ)$ were arbitrary,  flatness follows from {\rm Fact \ref{fact}(1)}.
\end{rem}

The rest of the paper is devoted to  the question of how the two dichotomies relate. First of all, they  yield  a subdivision of $\Lambda$ into four regions.


 
 \vspace{3em}
 
 \setlength{\unitlength}{2cm}
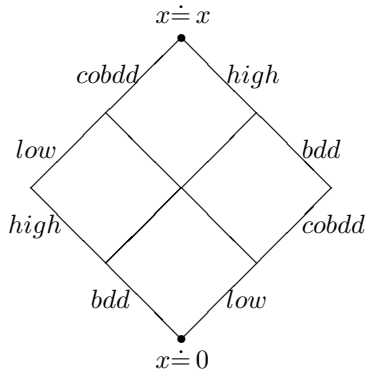
\begin{picture}(2,2)
  \put(0,1){\line(1,1){1}}
 \put(0,1){\line(1,-1){1}}
  \put(1,1){\line(-1,-1){.5}}
  \put(1,1){\line(-1,1){.5}}
   \put(1,1){\line(1,-1){.5}}
    \put(1,1){\line(1,1){.5}}
     \put(1,1){\line(-1,-1){.5}}
      \put(1,0){\line(1,1){1}}
      \put(1,2){\line(1,-1){1}}
      
       \put(1.3,1.7){$high$}
         \put(.3,1.7){$cobdd$}
              \put(-.15,.7){$high$}
                \put(1.8,.7){$cobdd$}
                 \put(.4,0.2){$bdd$}
                  \put(1.3,0.2){$low$}
               \put(-.1,1.2){$low$}
                  \put(1.8,1.2){$bdd$}

                   \put(1,2){\circle*{.05}}
                   \put(1,0){\circle*{.05}}
                   \put(.83,-.2){$x\dot=\, 0$}
                   \put(.83,2.1){$x\dot=\, x$}
 \end{picture}
 \vspace{2em}

The north region and the south region are, as indicated, never empty. But the east and west regions can be, as we are about to see.

\begin{rem}\label{empty}
 Note that the emptiness of  the  west region can be stated by {\em every high formula is cobounded} or by {\em every low formula is bounded}, but also by statements about their duals, like the ones that derive from {\rm Rem.\ref{reformulation}(5)}  or by {\em every high formula has a bounded dual} and by {\em every low formula has a cobounded dual}.
 Similarly for the east region.
\end{rem}

In \S\ref{dom} we show that the east region is empty precisely  when the ring under question is a domain. When the west region is empty---that  is the region of high-low formulas---seems a more elusive problem. 

\begin{rem}\label{hilosym}
 Elementary duality sends the north region on one side of the ring to the south region on the other, and vice versa. It sends the west region on one side to the west region on the other, and similarly for the east region. This implies that the figures in \S\S3 and 5  are two-sided: $_R\Lambda$ has the same shape as $\Lambda_R$ as far as the emptyness of west or east regions is concerned.
\end{rem}

A few words about types related to the four regions. Given a filter $\Phi$ and an ideal $\Psi$ in $\Lambda$, it is standard to consider the type $p_{\Phi/\Psi} := \Phi \cup \neg\Psi$---with the shorthand $\neg\Psi = \{\neg\psi\, :\, \psi\in\Psi\}$. A realization of $p_{\Phi/\Psi}$ in a module $M$ is an element $a$ that makes $p_{\Phi/\Psi}$  true, which means that $a\in \bigcap_{\phi\in\Phi} \, \phi(M) \setminus \bigcup_{\psi\in\Psi} \, \psi(M)$. The type $p_{\Phi/\Psi}$ is consistent precisely when it has such a realization (in some module). Clearly, if $\Phi$ and $\Psi$ have a formula in common, $p_{\Phi/\Psi}$ is inconsistent. The compactness theorem shows the converse:  if $\Phi\cap\Psi=\emptyset$, then $p_{\Phi/\Psi}$ is consistent. On the other hand, if $\Lambda = \Phi \cup \Psi$, it is clear that $p_{\Phi/\Psi}$ decides  every formula in $\Lambda$, i.e., $p_{\Phi/\Psi}$ is  complete. (It really is, in the technical sense of, say, \cite{H}, due
to the pp elimination in modules---a fact one may safely ignore in this study.)

\begin{rem}\label{types}
 \begin{enumerate}[\rm (1)]
 \item It is easy to see that the high formulas form a filter, denoted $\Phi_{hi}$, while the low formulas form an ideal, denoted $\Psi_{lo}$. Denote the corresponding type by $p_{hi/lo}$. It is consistent if and only if there is no high-low formula.
 \item It is not too hard to verify directly that the bounded (left) formulas form an ideal if and only if $R$ satisfies the left Ore condition\footnote{This is not the true Ore condition when the ring is not a domain.}: $Ra\cap Rb\not= 0$ for all nonzero scalars $a$ and $b$---which is to say that $_RR$ is uniform. (For necessity consider the module $R/Ra\oplus R/Rb$.) 
 In that case, $\Lambda^1$ is partitioned into $\Phi_{hi}$ and the ideal, $\Psi_{bdd}$, of bounded formulas. This yields  another consistent and complete type,   $p_{hi/bdd}$, provided $_RR$ is uniform. 
 \item Even easier is to verify that the set of cobounded (left) formulas form a filter, $\Phi_{cbdd}$,  if and only if the ring $R$ satisfies the right Ore condition, i.e., if  $R_R$ is uniform. In that case, Dichotomy II guarantees another  consistent and complete type,   $p_{cbdd/lo}$.
 \item By\/ {\rm Thm.\,\ref{Ore}} below, over a two-sided Ore domain, these three types are the same, for the simple reason that $\Phi_{hi}=\Phi_{cbdd}$ and 
 $\Psi_{lo}=\Psi_{bdd}$. (This type has been made use of in \cite{HP}, cf.\ \cite[\S8.2.1]{P2}.)
 \end{enumerate}
\end{rem}

\begin{exam}
  The uniform ring $R=\Bbb Z/4 \Bbb Z$ has (up to $\sim$) only four unary pp formulas, which form a chain: $x\eq 0 \leq 2|x \leq 2x\eq 0\leq x\eq x$. (And they are all different---consider $\Bbb Z/2 \Bbb Z \oplus \Bbb Z/4 \Bbb Z$.)
  The two in the middle are bounded and cobounded. All three types are distinct and equivalent to (isolated by) single formulas: $x\eq x \wedge \neg x\eq 0$ is equivalent to $p_{hi/lo}$ (which is, note, incomplete), $x\eq x \wedge \neg 2x\eq 0$ is equivalent to $p_{hi/bdd}$, and $2|x \wedge \neg x\eq 0$  is equivalent to $p_{cbdd/lo}$. This ring is self-injective, so possesses property $(\flat\sharp)$, to be introduced next.
\end{exam}

\section{The flat-sharp condition}\label{four}  Here we investigate rings with no high-low formulas, that is rings with empty west region as shown in the figure below, see Rem.\ref{empty} for other ways to say that.

 \vspace*{3em}
 
 \setlength{\unitlength}{2cm}
\begin{picture}(2,2)
  \put(1,2){\line(-1,-1){.5}}
   \put(1,0){\line(-1,1){.5}}
  \put(1,1){\line(-1,-1){.5}}
  \put(1,1){\line(-1,1){.5}}
   \put(1,1){\line(1,-1){.5}}
    \put(1,1){\line(1,1){.5}}
     \put(1,1){\line(-1,-1){.5}}
      \put(1,0){\line(1,1){1}}
      \put(1,2){\line(1,-1){1}}
      
       \put(1.3,1.7){$high$}
         \put(.3,1.7){$cobdd$}
                \put(1.8,.7){$cobdd$}
                 \put(.4,0.2){$bdd$}
                  \put(1.3,0.2){$low$}
                  \put(1.8,1.2){$bdd$}

                   \put(1,2){\circle*{.05}}
                   \put(1,0){\circle*{.05}}
                   \put(.83,-.2){$x\dot=\, 0$}
                   \put(.83,2.1){$x\dot=\, x$}
                   
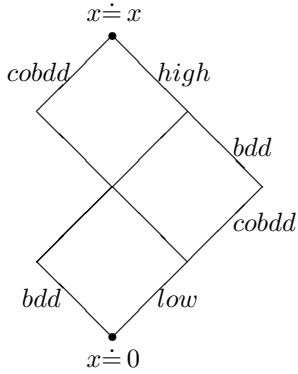
                   \put(0.2, -.5) {\small A lattice $\Lambda$ with no high-low formulas}
 \end{picture}
 \vspace{4em}
 
 This picture is left-right symmetric,  Rem.\ref{reformulation}(6) .
 
 \begin{definition}
 Call a module  \texttt{flat-sharp} or simply $(\flat\sharp)$ if it is nonzero, flat, and absolutely pure. We use $_R(\flat\sharp)$ to express that there be a $(\flat\sharp)$ left $R$-module.\
\end{definition}

\begin{rem}\label{flatsharp}
 \begin{enumerate}[\rm (1)]
 \item Over a ring with a $(\flat\sharp)$ module (whether on the left or on the right),  there is no (left or right) high-low formula, hence every low formula is bounded and every high formula is cobounded. 
  \item If the ring $R$  is left absolutely pure, then $_RR$ is a $(\flat\sharp)$ module, so $_R(\flat\sharp)$ holds. This applies to all (von Neumann) regular rings, and many more. (Of course, over regular rings {\em all} modules are flat and absolutely pure anyway.)
\item By\/ {\rm Fact \ref{fact}}, if a pp pair opens up in a left  $(\flat\sharp)$ module, it opens up in $_RR$ and its dual opens up in $R_R$.
\item In particular, if\/ $0 \not= \phi(A)\not= A$ in a left $(\flat\sharp)$ module $A$, then $\phi$ is bounded and cobounded (as is $2|x$ in $R=\Bbb Z/4 \Bbb Z$).
\end{enumerate}
 \end{rem}

\begin{quest}
Is every ring with no high-low formula $_R(\flat\sharp)$ or $(\flat\sharp)_R$?   

Is $_R(\flat\sharp)$ equivalent to $(\flat\sharp)_R$? 
\end{quest}

\begin{rem}
 The $(\flat\sharp)$ condition is two-sided for two-sided coherent rings, as follows  from Lambek's  and W\"urfel's  theorems on character modules, {\rm Fact \ref{Wur}}. 
 \end{rem}

This symmetry will be shown also for arbitrary domains in \S\ref{dom} below, over which $(\flat\sharp)$ turns out to be equivalent to the two-sided Ore condition. Further,  the domains with no high-low formulas are exactly the $(\flat\sharp)$ domains. So, over domains, both questions have an affirmative answer.

I list some properties of the ring---in order of decreasing strength---which all  entail $_R(\flat\sharp)$ and thus the non-existence of high-low formulas. 
\vspace{3mm}

\noindent
(QF) Every injective $R$-module is projective.\vspace{1mm}

\noindent
$_R$(IF) Every injective left $R$-module is flat.\vspace{1mm}

\noindent
$_R$(H) The injective envelope (or hull) of every flat left $R$-module is flat.\vspace{1mm}

\noindent
$_R$(H$_{reg}$) The injective envelope of the ring, as a left module over itself, is flat.\vspace{1mm}

\noindent

Similarly on the right, where  the subscript $R$ is omitted altogether when  the property is required to be true on both sides. Another interesting condition is\vspace{3mm}

\noindent
(FA)$_R$ Every flat right $R$-module is absolutely pure\vspace{1mm}.

\begin{rem}\label{FA}
 \begin{enumerate}[\rm (1)]
 \item
 $_R$(IF)$\implies$(FA)$_R \implies(\flat\sharp)_R$. (The last implication is trivial, the first follows from Lambek and Würfel I and is stated in \cite[Bem.4.7]{Wu}.)
\item Thus $_R$(IF)$\implies$$(\flat\sharp)$ (the two-sided property).

This follows from (1) together with the trivial implication $_R$(IF)$\implies_R(\flat\sharp)$.
\item \cite[Prop.2.4]{GR}. If $R$ is left or right coherent, then $_R$(IF)$\iff$(FA)$_R$ and (IF)$_R$$\iff$$_R$(FA). 

For left coherent $R$, the first equivalence follows from Würfel II; for right coherent (or even just right a-coherent $R$, {\rm Def.\,\ref{a-}}) from Thm.\,\ref{Colby1} below. The second equivalence follows by symmetry.
\end{enumerate}
\end{rem}

\begin{rem}
\begin{enumerate}[\rm (1)]
\item It is well known that {\rm (QF)} is a two-sided condition equivalent to $R$ being Quasi-Frobenius.  
\item Rings with {\rm $_R$(IF)} are commonly called  \texttt{left IF rings}. 

As absolutely pure modules are pure in their injective envelope and pure submodules of flats are flat, it is clear that, moreover, all absolutely pure left modules over such a ring are flat  (see the next section for more detail). 

\item {\rm $_R$(IF)} is not a symmetric property: \cite[Example 2]{Co} is an example of a two-sided coherent ring that is absolutely pure only on one side, see also \cite[Example 5.46]{NY}. Such a ring is properly one-sided IF by\/ {\rm Thm.\,\ref{Colby1}} below.
\item The closure condition {\rm (H)} on the class of flat (left) modules was introduced  in \cite{CE} and further investigated in \cite{E}. \cite[Thm.3]{CE} characterizes  the commutative noetherian rings with {\rm (H)}. \vspace{1mm}

\item  {\rm $_R$(H$_{reg}$)} and  {\rm $_R$(H)} are equivalent  for left nonsingular rings of finite Goldie dimension, \cite[remarks after Thm.3]{CE}, and for commutative noetherian rings, \cite[Thm.]{E}; mind a misprint in condition b: over such a ring, {\rm  (H)} is equivalent to: the flat cover of every injective is\/ {\em injective}.)
\item This leads to another interesting property, which also implies $(\flat\sharp)$ (on the corresponding side):

{\rm (C)}. Every flat cover of an injective is injective.

\item In \cite[last line]{CE} it is mentioned that $k[x, y]/ (xy)$, with $k$ a field, is an example of a commutative ring with {\rm (H)} that is not a domain.  Another one is $R=\Bbb Z/4 \Bbb Z$.
\item Over a right nonsingular ring, $\E(R_R)$ carries a ring structure extending $R$, cf.\ \cite[Ch.2]{G}. Thus this module is also a left $R$-module. \cite[Thm.3.6]{G} states conditions that are equivalent to flatness of $_R\E(R_R)$. So, when the ring is commutative, these turn into conditions for {\rm (H$_{reg}$)}. In particular, a commutative nonsingular and coherent ring satisfies  {\rm (H$_{reg}$)},  \cite[Cor.3.6]{G}.
\item If a commutative noetherian ring, $R$, has property {\rm  (H)}, then so do $R[x]$, $R[[x]]$ and $S^{-1}R$ for every multiplicative set $S\subset R$, \cite[Prop.]{E}.
\end{enumerate}
\end{rem}

\begin{quest}
 The hereditary nature of (H) exhibited in the previous item  raises the same questions for $(\flat\sharp)$. 
\end{quest}

\section{Some classes of rings}
\subsection{IF rings}\label{IF} We continue the study of this special class of rings over which there are no high-low formulas. Würfel proved the equivalence of (1), (3) and a version of (7) (with almost coherence instead of a-coherence) below, \cite[Satz 4.3]{Wu}, and 
Colby \cite[Thm.1]{Co} proved  the equivalence of (1) through (4)  and of right $\aleph_0$-injectivity and another weak form of right coherence. Colby found himself unable to find a left IF ring that was not outright right coherent, \cite[end of \S 2]{Co}. I tried to proof that there is no such ring, but wasn't able to do that either. Nevertheless I add some conditions that seem to get closer to doing so and may be of interest in their own right---pp conditions characterizing $_R$(IF). For the convenience of the reader, I include a complete proof, of which the brief (2)$\implies$(3) and 
 (4)$\implies$(1) are duplicated from Colby's paper. The proof of (4)$\implies$(7) employs the idea of proof of Würfel II given in \cite[Thm.4.4]{PRZ}.

\begin{definition}\label{a-}
Call a ring $R$ \texttt{right a-coherent} if the right ideal $I:= \phi(_RR)$ has left annihilator $\lann(I)=\lann(I_0)$ for some fg right subideal $I_0\subseteq I$, for every unary left pp formula  $\phi$.
\end{definition}

\begin{rem}
 The prefix \texttt{a-} stands for `annihilator.' The notion is a weakening of right coherence: a ring is right coherent ring iff every  pp definable  right ideal is finitely generated, \cite[Thm.2.3.19]{P2}, so that, moreover, $I=I_0$.
\end{rem}

\begin{thm}\label{Colby1} The following are equivalent for any ring $R$.
 \begin{enumerate}[\rm (i)]
 \item $_R$(IF), i.e., every injective left $R$-module is flat.
\item The injective envelope of every fp left $R$-module is flat. 
\item Every fp left $R$-module embeds into a free left $R$-module.
\item The character module of every free right $R$-module is a flat left $R$-module.
\item Every left pp formula $\phi$ is $_R\sharp$-equivalent to a divisibility formula $\mtx A\, |\, \br x$ below $\phi$ (in $_R\Lambda$).
\item Every right pp formula $\psi$ is $\flat_R$-equivalent to an annihilation  formula $\br x\, \mtx A \eq \br 0$ above $\psi$ (in $\Lambda_R$).
\item $R$ is right absolutely pure and right a-coherent.
\end{enumerate}
It suffices to consider unary pp formulas in the above.
\end{thm}
\begin{proof}
 (1)$\implies$(2) is trivial. 
 
 \noindent
 (2)$\implies$(3). Let $P\in \Rmod$. By (2), $\E(P)$ is flat. Then, by Lazard's Theorem, the embedding $i: P \to \E(P)$ factors through a free module $F$, i.e., there are $j: P \to F$ and $k: F\to \E(P)$ with $kj = i$. But as $i$ is an embedding, so is $j$, which shows (3). 
 
  \noindent
  (3)$\implies$(4). Let now $F$ be a free right $R$ module. To prove, $F^+$ is flat, we show that, given a unary left pp formula $\phi$, we have $\phi(F^+)=\phi(_RR)F^+$. Only the inclusion from left to right is nontrivial. So consider $b\in \phi(F^+)$ and  pick a fp free realization $(P, a)$ of $\phi$. Then there is a map $f: (P, a)\to (F^+, b)$, and by  (3), $P$ is contained in a free left $R$-module $L$. 
  As $F$ is flat, $(F^+)$ is injective, by Lambek's Theorem. Hence $f$ factors through $L$, this yielding a map $g: L\to (F^+)$ with $a \mapsto b$. Note that also $a\in \phi(L)$, whence, by flatness of $L$, this element is in $\phi(_RR)L$. Application of $g$ shows $b\in \phi(_RR)F^+$, as desired.
  
   \noindent
   (4)$\implies$(1).  Let $M$ be an injective left $R$-module. Write $M^+$ as an epimorphic image of a free right $R$-module $F$. It is easy to check that this yields an embedding of $M^{++}$ into $F^+$. Thus $M$ is embedded in  $F^+$, hence splits off by its choice.  $F^+$ being flat by (4), this implies that so is $M$.
   
Thus (1) through (4) are equivalent. So are (5) and (6) (by elementary duality). (5)$\implies$(1) is straightforward: let $\phi$ be unary and $_R\sharp$-equivalent to a divisibility formula $\psi:= \mtx b\, |\,  x$ below $\phi$ and let $E\in \RMod$ be injective; to prove $\phi(E)\subseteq \phi(_RR)E$ it suffices to verify $\psi(E)\subseteq \psi(_RR)E$ (for $\psi\leq\phi$), which is obvious, as $\mtx b\, | \,  \mtx b$ in $_RR$. So for (1) through (6) to be equivalent it remains to prove

 \noindent
   (3)$\implies$(5), which can be done locally thus. Suppose $P\in\Rmod$ is a submodule of a free left $R$-module $F$. We are going to show that every pp formula $\phi=\phi(\br x)$ freely realized in $P$ is $_R\sharp$-equivalent to a divisibility formula $\psi := \mtx A\, |\, \br x$ with $\psi\leq\phi$. Let $\phi$ be freely realized by a tuple $\br a$ in $P$. Being existential, $\phi$ is also realized by  $\br a$ in $F$. By \cite[Lemma 1.2.29]{P2}, in the free module $F$,  this tuple freely realizes a divisibility formula $\psi$ as above. Then  $\psi$ is below $\phi$ and it remains to verify that it is $_R\sharp$-implied by $\phi$.  So let $\br c\in \phi(E)$ with $E$ injective. By choice of $\phi$, there is a map $f: (P, \br a)\to (E, \br c)$, which, by injectivity, can be extended to $g: (F, \br a) \to  (E, \br c)$, whence $\br c\in \psi(E)$, as desired. 
   
   Note, we used (5) (or (6)) only for unary $\phi$ in the previous argument, which confirms the extra clause in the theorem (the same can be said about the remaining arguments). To prove that any of these equivalent conditions imply the remaining ones, first note that, by Rem.\ref{FA}(1), any left IF ring is right absolutely pure.
   
   Next suppose (4) (and (1)) holds and assume (7) does not. Let us, just for simplicity, first assume that $I$ is countably generated, by $r_0, r_1, \ldots$, say. By assumption w.l.o.g. there are elements $r_n\in I$ $(n<\omega)$ such that the left annihilators of $I_n=\{r_i\, : \, i<n\}$ (with $n>0$) form a properly descending chain (above $\lann(I)$). Pick $a_n\in \lann(I_n)\setminus\lann(I_{n+1})$ and a character $g_n\in (R_R)^+$ with $\lann(I_{n+1})\, g_n=0$, but $a_n\, g_n \not= 0$. Let $F= R_R^{(\omega)}$ the free right $R$-module of rank $\omega$. As $F^+ = ((R_R)^+)^\omega$, the map $g=\langle g_n \rangle$ is a character on $F$ with $\ann_F(I)\,g=0$ (for $\ann_F(I)= \lann(I)^{(\omega)}\subseteq F$). As $R$ is right absolutely pure and $I:= \phi(_RR)$, we have $\D\phi(R_R) = \lann(I)$, hence all $\D\phi(R_R)\, g_n =0$ in $R_R^+$ and thus $\D\phi(R_R)\, g =0$ in $F^+$,  which shows that $g\in \phi(F^+)$. Invoking (4), 
  $g$ is a linear combination $\sum s^j f^j$ with $s^j\in I$ and $f^j\in F^+$. All these finitely many $s^j$ are contained in  $I_N$ for some $N<\omega$. This leads to a contradiction as follows. 
  
  Let $c_n$ be the element of $F$ whose $n$th coordinate is $a_n$, while all the others are $0$. Then $c_n\, g = a_n \, g_n \not= 0$ for all $n$.  In contrast, $c_N\,g = a_N g_N= a_N (s^j f^j_N)$ (where the latter is the restriction of $f^j$ to the $N$th coordinate of $F$), which is $0$, since $a_N\, s^j \in a_N\,I_N =0$, by the choice of  $a_N$. This contradiction completes the proof of (4)$\implies$(7) for countably generated $I$.  
If the generators of $I$ are indexed by some uncountable cardinal $\alpha$, simply replace $\omega$ by $\alpha$ everywhere (and work with $F$ of rank $\alpha$ instead).  

 \noindent   
(7)$\implies$(6).  It suffices to prove this for unary pp formulas (see above). Every unary right pp formula is (equivalent to one) of the form $\D\phi$ for some unary left pp formula $\phi$. Setting $I= \phi(_RR)$ again,  we have, as before, that $\lann(I)=\D\phi(R_R)$, for $R_R$ is absolutely pure by hypothesis. Then, by (7), $\D\phi(R_R)=\lann(I_n)$ of a fg subideal $I_n$ of $I$. Let $\psi$ be the unary annihilation formula that defines $\lann(I_n)$ in $R_R$. Then $\D\phi$ is equivalent to $\psi$ in $R_R$. But for every flat right module $F$, we have $\D\phi(F)= F\,\D\phi(R_R)= F\,\psi(R_R) =\psi(F)$, hence the two formulas are equivalent even in all flat right $R$-modules, as desired.
\end{proof}

\begin{quest}[Colby]\label{Colbyquest}
 Is there  a left IF ring that is not right coherent?
\end{quest}

Combining right a-coherence with right $\aleph_0$-injectivity, we obtain a pp version of \cite[Cor.1]{Co}.

\begin{prop}\label{double}
 Suppose $R$ is right $\aleph_0$-injective and right a-coherent.
 \begin{enumerate}[\rm (1)] 
  \item If $A$ and $B$ are right ideals that are pp definable in $_RR$, then\\
    $\mtx l(A\cap B) = \mtx l(A) + \mtx l(B)$.
      \item $\mtx l\mtx r (C)=C$, for every f.g.\ left ideal $C$.
    \end{enumerate}
\end{prop}
    Note that $\mtx l(A\cap B)$ is   equal to $\mtx l(C_0)$ for some f.g.\ right ideal $C_0\subseteq A\cap B$, for $A\cap B$ is also pp definable.

\begin{proof} 
\cite[Thm.12.4.2(a)]{K} states (1) for arbitrary right ideals in a right injective ring $R$.  It is easy to see that its proof yields the same for f.g.\ ideals in a right  $\aleph_0$-injective ring $R$. Applying this to 
 f.g.\ right ideals $A_0\subseteq A$ and $B_0\subseteq B$ with $\mtx l(A) = \mtx l(A_0)$ and  $\mtx l(B) = \mtx l(B_0)$ (which exist by a-coherence), we have $\mtx l(A\cap B) \subseteq \mtx l(A_0\cap B_0) = \mtx l(A_0) + \mtx l(B_0) = \mtx l(A) + \mtx l(B) \subseteq \mtx l(A\cap B)$, so all are equal as desired.
 
It is well known that (2) for $C = Rc$ (principal left ideal) is equivalent to right $P$-injectivity of $R$, see for example [ibid] or \cite[3.17]{LMR} or \cite[Lemma 5.1]{NY}. So we certainly have that in the right $\aleph_0$-injective ring $R$. To extend this to arbitrary f.g.\ left ideals $C= \sum Rc_i$, first write $\mtx r (C) = \bigcap \mtx r (Rc_i)$ and then notice that each $\mtx r (Rc_i)$ is a right pp definable ideal, so that iterated application of (1) concludes the proof of (2).
\end{proof}

\begin{rem}\label{K}
 The proof of \cite[Thm.12.4.2(b)]{K} shows that {\rm (2)} together with {\rm (1)} for f.g.\ right ideals imply right $\aleph_0$-injectivity of $R$.
\end{rem}

\begin{quest}\label{Kasch}
 Does right $\aleph_0$-injectivity and right a-coherence together  imply right absolute purity of $R$ (and thus  $_R$(IF))? 
\end{quest}

(This would follow from the theorem and Colby's theorem if right a-coherence implied right weak coherence of \cite[Thm.1(6)]{Co}.)\\

One partial answer is given next. For another one, see \S\ref{a-coh}.

\begin{cor}\label{abspure}

 \begin{enumerate}[\rm (1)] 
\item Suppose  {\rm (H$_{reg}$)}$_R$  holds (i.e., $\E(R_R)$ is flat).

\noindent
 Then, if $R$ is right $\aleph_0$-injective and right a-coherent,  $R$ is right absolutely pure (hence left IF).
\item If $R$ is right $\aleph_0$-injective and two-sided a-coherent, then $R$ is left coherent\\ (hence "$\aleph_0$-injective=absolutely pure" in $\RMod$).
     \end{enumerate}
\end{cor}
\begin{proof}
 (1). All we need to show is that $R_R$ is pure in $E:=\E(R_R)$, i.e., $R\,\cap\, \phi(E)\subseteq \phi(R_R)$ for every right unary pp formula $\phi$. By flatness, $\phi(E) = E\, \phi(R_R)$, so every $r\in R\cap\phi(E)$ can be written as $r=\sum e_i\, r_i$ with $e_i\in E$ and $r_i\in \phi(R_R)$. Let $C$ be the left ideal generated by those $r_i$. Then $\mtx l\mtx r (C)=C$ by (2). But clearly $r$ is contained in the left hand side, so $r\in C\subseteq \phi(R_R)$.
 
 (2). By \cite[Thm.2.3.19]{P2} it suffices to verify that every left ideal which is pp definable in $R_R$ is f.g. Let $C$ be such an ideal. By left a-coherence $\mtx r (C)= \mtx r (C_0)$ for some f.g.\  left subideal $C_0\subseteq C$. By Prop.\ref{double}(2) above, $C\subseteq \mtx l\mtx r (C)= \mtx l\mtx r (C_0)= C_0 \subseteq C$, so all are equal.
\end{proof}

We can now derive the characterization of two-sided  IF rings  from the one-sided theorem by means of the proposition (or (1) of the previous corollary). (i)$\iff$(ii) below is part of \cite[Thm.2]{Co}, but all of it---except  for the minor variant (iii), which, recall, implies (ii) under coherence---is already contained in \cite[Folgerung 4.6]{Wu}, who attributes it to \cite[Th.2]{Sab}. (A precursor, assuming coherence at the outset, is \cite[Prop.4.2]{Ste}.) 

\begin{cor}[Sabbagh]
 The following are equivalent for any ring $R$.
 \begin{enumerate}[\rm (i)]
 \item $R$  is (two-sided) IF.
  \item $R$ is (two-sided)  absolutely pure and coherent.
 \item $R$ is (two-sided) $\aleph_0$-injective and coherent.
 \item $_R${\rm (IF)} and  $_R${\rm (FA)}, i.e., $_R\flat = {_R\sharp}$.
 \item {\rm (IF)}$_R$ and  {\rm (FA)}$_R$, i.e., $\flat_R = \sharp_R$.
     \end{enumerate}
\end{cor}
\begin{proof}
Recall, absolutely purity implies $\aleph_0$-injectivity, and  they are the same over coherent (on that side) rings,  by \cite[Prop.3.23]{ES}, as mentioned at the end of \S\ref{1}. Thus, (ii)$\iff$(iii).  As a-coherence is implied by coherence,  (i)$\iff$(ii) follows from  Thm.\ref{Colby1} and the first part, (1), just proved. Now (iv) implies right coherence (by Chase's theorem and the fact that products of flats are absolutely pure, so flat again). So Rem.\,\ref{FA}(3) yields (iv)$\implies$(v). The converse follows by symmetry.  The implication (iv)$\implies$(v) also entails (i). Finally, (i) implies (iv) (and (v)) by Rem.\,\ref{FA}(1).
\end{proof}

\begin{rem}
 \begin{enumerate}[\rm (1)] 
 \item As products preserve also $\aleph_0$-injectivity, \cite[Prop.3.9]{ES}, the same proof shows that another equivalent condition is "flat =  $\aleph_0$-injective" on either side of the ring. This is \cite[Prop.2.3]{GR}.
 \item \cite[48.8]{W} states some of these and other equivalent conditions for IF rings. (Beware, what is called  QR ring there is just a two-sided  absolutely pure ring.) One of the equivalent conditions, {\rm 48.8(d)},  consists of (i), left coherence of $R$, (ii)\, = {$_R$(FA)}, and (iii), a condition equivalent to {\rm $_R$(IF)}. Sabbagh's result above shows that (i) (left coherence) is superfluous in {\rm 48.8(d)}. 

  \end{enumerate}
\end{rem}

More on IF rings can be found in \cite{GR}, \cite[Ch.6]{Fa} and \cite[\S 48]{W}.

\subsection{a-coherent rings revisited}\label{a-coh} Question \ref{Kasch} stated before Cor.\,\ref{abspure} about whether right $\aleph_0$-injectivity implies right absolute purity of a right a-coherent ring has been given a partial answer in said corollary. We are now heading for another partial positive answer.

For the remainder of the section, let $\Phi$ be the set of unary left pp formulas $\phi$ such that $\D\phi(R_R) \subseteq \mtx l(I)$ (and hence equal), where $I = \phi(_RR)$, and let $\D\Phi$ be that of all of its duals. Recall, 
$R$ is right absolutely pure if and only if $\Phi$ consists of {\em all}\/ unary left pp formulas (iff $\D\Phi$ consists of all unary right pp formulas).

\begin{lem} Let $R$ be any ring. 
\begin{enumerate}[\rm (i)]
 \item $\Phi$ is closed under multiples and $\D\Phi$ is closed under inverses.
 \item $\Phi$ is closed under sum and\/ $\D\Phi$ is closed under conjunction (meet).
  \end{enumerate}
\end{lem}
\begin{proof}
(1).        Let $\psi\in\Phi$ and $J:= \psi(R_R)$ as before. By hypothesis, $\D\psi(R_R)=\mtx l(J)$. If $\phi= b\psi$, then $I:=\phi(R_R)= bJ$ and $\D\phi= (\D\psi)b^{-1}$ by Lemma \ref{dualmult}, hence $\D\phi(R_R)=\{ r\in R\, : \,  rb\in \D \psi(R_R)\} = \{ r\in R\, : \,  rb\in \mtx l(J) \}=   \mtx l(bJ) =   \mtx l(I)$. This proves the first statement. As every inverse of a formula from $\D\Phi$ is the dual of a mutiple from $\Phi$, by the same lemma, the second statement is equivalent to the first.

(2). Again, both statements are equivalent, since $D$ is an anti-isomorphism between $_R\Lambda$ and $\Lambda_R$. Consider finitely many formulas $\phi_i\in\Phi$ and let $\phi$ be their sum. Let their left annihilators in $R$ be $I_i$ and $I$, resp. Then $\D\phi=\bigwedge \D\phi_i$, hence $\D\phi(R_R)=\bigcap \D\phi_i(R_R)= \bigcap \mtx l(I_i) = \mtx l(\sum I_i)= l(I)$.
\end{proof}

\begin{lem} If $R$ is a right $\aleph_0$-injective and right a-coherent ring, then
$\Phi$ contains all multiples of formulas of the form  $\mtx a\,x \eq \br 0$ (where $\mtx a$ is a column vector over $R$), and $\D\Phi$ contains all formulas of the form $\mtx a\,|^{op}\,xb$, where $b\in R$.
\end{lem}
\begin{proof}
 Let $\psi$ denote the formula $\mtx a\,x \eq \br 0$. Set $J:= \psi(R_R)$ and let $C$ be the (f.g.) left ideal generated by the entries of $\mtx a$. Then $J= \mtx r(C)$, hence, by Prop.\,\ref{double}(2), $\mtx l(J)= C$. But $\D\psi = \mtx a\,|^{op}\, x$, so $\D\psi(R_R)=C$. This proves that $\psi$ (and all its multiples) are in $\Phi$. Hence $\mtx a\,|^{op}\,x\in \D\Phi$.
 
 Let now $\theta$ be the formula $\mtx a\,|^{op}\,xb$, where $b\in R$. This is the $b$-inverse of $\D\psi$, which is, by Lemma \ref{dualmult}, the dual of the $b$-multiple of $\psi$, which has just been shown to be in $\Phi$. So $\theta\in\D\Phi$.
\end{proof}

This can be put together to get an affirmative answer to  Question \ref{Kasch} in the case of 
\texttt{left Warfield ring}s. These are the rings all of whose left f.p.\ modules are direct summands of direct sums of cyclic f.p.\ modules,  see \cite[p.2139]{PPR} or \cite[p.75]{P2}.

\begin{prop}\label{Warfield}
 Suppose $R$ is a left Warfield ring.
 
 If $R$ is right $\aleph_0$-injective and right a-coherent,  $R$ is right absolutely pure (hence left IF).
\end{prop}
\begin{proof}
Let $\cal P$ be the set of all finite column vectors over $R$. \cite[Thm.25]{PPR} says among other things that $R$ is left Warfield iff every right pp formula is equivalent to a finite conjunction of formulas of the form $\mtx A\,|^{op}\, \br x \mtx B$ with $\mtx A \in \cal P$. By the choice of $\cal P$ 
we have that $\mtx A$ is, say $k \times 1$. Therefore $\mtx B$ must be a column vector too. However,  we are only interested in unary such formulas, so $\mtx B$ is actually just a scalar $b\in R$. Thus all unary right pp formulas are conjunctions of formulas of the form shown in the previous lemma to be in $\D\Phi$. But this set is closed under conjunction, so it actually contains all unary pp formulas (up to $\sim$). Therefore $R$ is right  absolutely pure.
\end{proof}

\subsection{Nonsingular rings}\label{nonsing} Recall, a module is \texttt{nonsingular} if no  annihilator of a nonzero element is essential in $_RR$. Similarly for right modules. 
$R$ is \texttt{right nonsingular} if $R_R$ is nonsingular.  
Clearly, domains are (right and left) nonsingular. For right nonsingular rings, the injective hull $\E(R_R)$ of the right module $R_R$ plays the `role of the fraction field' of a commutative domain. It forms a ring extending the ring $R$, the \texttt{maximal right quotient ring} of $R$, henceforth denoted by $Q$. This ring is (von Neumann) regular and right self-injective, \cite[Cor.2.31]{G}. Further,  the extension $R_R \subseteq Q_R$ is not only essential (as it always is), but even rational. An extension of right $R$-modules $A\subseteq B$ is  \texttt{rational} (or  \texttt{dense}), in symbols,  $A\rat B$, if for every $b_0, b_1\in B$ with $b_1\not=0$ there is a ring element $s\in R$ such that  $b_1 s \not=0$ and $b_0 s\in A$, \cite[Prop.2.25]{G} or \cite[Def.8.2]{LMR}. Iterating  (and multiplying the ring elements on the right) one can extend this easily as follows. 


\begin{rem}\label{rat}
Suppose $A\rat B$ is a rational extension  of right $R$-modules. 
 
For every $b_0, \ldots, b_n\in B$ with $b_n\not=0$ there is a ring element $s\in R$ such that $b_n s \not=0$ and $b_i s\in A$ for all $i<n$.
\end{rem}

It is easy to see that a rational extension is essential (but not v.v.---consider the group of $2$ elements inside the cyclic group of $4$ elements). 
In nonsingular modules, the converse is true. Every essential submodule of a nonsingular module is rational. Further, over a right nonsingular ring $R$, the right module $Q_R$ is nonsingular, 
\cite[Prop.2.9 and Lemma 2.24(b)]{G}. So, essential submodules and rational submodules of  $Q_R$ are the same thing. 

\begin{rem}
Given a ring extension $R\subseteq Q$, we have $R_R \rat Q_R$ if and only if, for every $q_0, q_1\in Q$ with $q_1\not=0$, the  system of equations
$q_1 y \eq x \wedge q_0 y \eq z$ has a solution $(r, s, t)$ in $R$ with $r\not= 0$.
\end{rem}

The above system of equations is a left (quantifier-free) pp formula, but beware, it is in $_Q\Lambda$ rather than in $_R\Lambda$. For unary left pp formulas in $_R\Lambda$, there is a similar useful fact of much greater generality (generalizing Lemma \ref{fract}).

\begin{lem}\label{Q}
 Suppose  $R\subseteq Q$ is a ring extension and $R_R \rat Q_R$.
 
(This is the case when $R$ is a right nonsingular  and $Q$ is the maximal  right quotient ring of $R$.) 

If $\phi\in {_R\Lambda}$, a unary left pp formula over $R$,
 then $\phi(_RR)\ess \phi(_RQ)$ (as right $R$-modules). In particular, 
  $\phi$  is low if and only if $\phi(_RQ)=0$   if and only if it is low in $_Q\Lambda$.
 
\end{lem}
\begin{proof} As $\phi$ is existential, $\phi(_RR)\subseteq \phi(_RQ)$. So we may assume that $\phi$ has a nontrivial solution $q\in {_RQ}$. To show that $qR\cap \phi(_RR)\not= 0$, write $\phi$ as $\mtx A | \mtx b x$ (with $\mtx A \in {^mR^n}$ and $\mtx b\in {^mR}$)  and let   $q_0, \ldots, q_{n-1} \in Q$ witness the truth of $\phi(q)$ in $_RQ$.  
As $R_R$ is essential in $Q_R$, there is  $r\in R$ such that $0\not= qr \in R$. Now set $q_n := qr$ and  choose $s$ as in Rem.\ref{rat} (for $A=R_R$ and $b$ and $B$ replaced by $q$ and $Q$). Then $qrs\in qR$ is a nontrivial solution of $\phi$ in $_RR$, as desired. 
\end{proof}

This yields another partial answer to Question \ref{quest1}. Note, for a right nonsingular ring $R$, in general $Q$ is only injective in $\ModR$ (and flat in $\RMod$---both on the `wrong' side).

\begin{cor}\label{commnonsing} Let  $R$ and  $Q$ be as in the lemma and suppose   $_RQ$ is injective (or just absolutely pure).
Then $R$ has no high-low formula.

In particular, no commutative nonsingular ring has high-low formulas.
\end{cor}
 
Next I state an elementwise  reformulation of an essentiality criterion from \cite[\S 1.D]{G} and one for essential closedness of right ideals it implies.

\begin{fact}\label{ess}
\begin{enumerate}[\rm (1)]
\item Suppose $A$ is a submodule of  a nonsingular right $R$-module $B$.

$A\ess B$ if and only if $b^{-1} A \ess R_R$ for every $b\in B$. 

\item
Only $b\in B\setminus A$ is meaningful in (1), and if $B = A + bR$ (is cyclic over $A$ and nonsingular), then $A\ess B$ if and only if $b^{-1} A \ess R_R$.

%
\item Suppose $I$ is a right ideal of a right nonsingular ring $R$.

$I$ is essentially closed in $R_R$ (i.e., not essential in any right ideal $J\supset I$) if and only if $r^{-1}I \not\ess R_R$  for all $r\in R\setminus I$.

\end{enumerate}
\end{fact}
\begin{proof}
(1) $\implies$ follows from \cite[Prop.1.1(c)]{G}. For the converse, note that $b^{-1} A \ess R_R$ says the right annihilator of $b+A$ in $B/A$ is essential in $R_R$. Hence the right hand side implies that $B/A$ is singular, which is equivalent to $A\ess B$ by \cite[Prop.1.21]{G}.

(2) That it suffices in (1)$\impliedby$ to check generators $b$ over $A$  follows from $\implies$, for if $b^{-1} A \ess R_R$ then so is $s^{-1}b^{-1} A \ess R_R$ and, as  $s^{-1}b^{-1} A = (rs)^{-1} A$, this yields the right hand side for all $b'\in B = A + bR$.

(3) $I$ is essentially closed in $R_R$ iff $I\not\ess I+rR$ for any $r\in R\setminus I$, which  is equivalent to $r^{-1} I \not\ess R_R$---by (2) applied to the nonsingular  $B= I+rR$.   
\end{proof}

Before applying this to formulas,  some terminology.

\begin{definition}
 Call a unary (left) pp formula $\phi$ \texttt{essential}, resp., \texttt{closed}, if $\phi(_RR)$ is essential, resp., essentially closed, in $R_R$. 
 \end{definition}

\begin{rem}\label{essphi}  $\phi$ is essential (over any ring) if and only if for no\/ $t \in R^0$ is $\phi\wedge t|x$   low if and only if\/ $\phi\wedge t|x$ is cobounded for
all\/ $t \in R^0$. In particular, essential formulas are cobounded.
\end{rem}
 
 \begin{lem}
 Suppose  $R$ is a right nonsingular ring.
 
\begin{enumerate}[\rm (1)]
\item All nontrivial multiples of an essential (left) formula are cobounded.
\item Essential (left) formulas are high (and cobounded, i.e.,north formulas). 
\item If $\phi(_RQ)\supseteq R$ (for which it suffices that\/ $1\in \phi(_RQ)$), then $\phi$ is essential, high, and cobounded.
 \end{enumerate}
\end{lem}
\begin{proof} (1). If  a multiple $t\phi$ is low, then $\phi(_RR)\subseteq \rann(t)$, hence either $t=0$ or $\phi$ is not essential. 

(2). Let $\phi$ be essential. We know already, over any ring,  $\phi$ is cobounded. If it were bounded, by $t\eq 0$ say, then the multiple $t\phi$ would be low, contradicting (1).

Under the hypothesis of (3),  $\phi(_RR)\subseteq R \subseteq \phi(_RQ)$, hence $\phi$ is essential by  Lemma
\ref{Q}. 
\end{proof}

Note, $Q$ is injective on the `wrong' side, so in (3) above we cannot infer that $\phi(_RQ)=Q$, even though $\phi$ is high. All we know is that $\phi(\E(_RR)) = \E(_RR)$, and similarly for all other injective left $R$-modules.
 
Next I  reformulate  Fact \ref{ess}(3) for pp formulas. Duplications with contrapositives are intended. 

\begin{lem}
 Suppose $\phi$ is a left pp formula  over a right nonsingular ring $R$.
 
\begin{enumerate}[\rm (1)]
 \item $\phi$ is closed if and only if no formula of the form $r^{-1}\phi$ with $r\in R\setminus \phi(_RR)$ is essential.
  \item $\phi$ is closed if and only if, for all $r\in R\setminus \phi(_RR)$, there is\/ $t \in R^0$ such that the conjunction  $r^{-1}\phi \wedge  t|x$ is low.
\item $\phi$ is not closed if and only if some formula of the form $r^{-1}\phi$ with $r\in R\setminus \phi(_RR)$ is essential.
  \item $\phi$ is not closed if and only if there is $r\in R\setminus \phi(_RR)$ such that $r^{-1} \phi\wedge t|x$ is cobounded  for all\/  $t \in R^0$.
 
 \end{enumerate}
\end{lem}
\begin{proof}  
(1)  is a special case of Fact \ref{ess}(3). Together with Rem.\,\ref{essphi}, it implies (2). The remaining items are contrapositives. 
\end{proof}

 \subsection{Reduced rings}\label{red}

First we state an immediate consequence of Lemma \ref{rinverse}.

\begin{cor} Suppose $r\in R$ is not nilpotent and $\phi$ is a unary pp formula.

 \begin{enumerate}[\rm (1)]
 \item If  
$r^{n}M\,\cap\,\phi(M)\not=0$  and $r \phi(M) = 0$, with $M$   a faithful $R$-module,    then $\phi(M)\subset r^{-1} \phi(M)\subset  \ldots \subset r^{-{n}} \phi(M)$ forms a properly ascending chain of $n+1$ pp subgroups of $M$.  
 \item If $\phi$ is an essential formula such that $r\phi$ is low, then  $\phi(_RR)\subset r^{-1} \phi(_RR)\subset r^{-2} \phi(_RR)\subset r^{-3}\phi(_RR)\subset \ldots$ is an infinite properly ascending chain of right ideals.  
 \item Every essential pp formula $\phi$ that is bounded by a non-nilpotent ring element gives rise to an infinte ascending chain of pp definable (right) ideals above the right ideal $\phi(_RR)$.\qed
 \end{enumerate} 
\end{cor}

\begin{rem}
 \begin{enumerate}[\rm (1)]
 \item Given $r, s\in R$ with $sr=0$, we have $M[s]\not= 0$ or $M[r]\not= 0$ in every nonzero $R$-module $M$.
 \item If $R$ is not a domain, then, in every nonzero $R$-module, $M[r]\not= 0$ for some $r\in R^0$.
 \item If there is a faithful pp-simple $R$-module, then $R$ is a domain.
 \item 
 If  an absolutely pure module $M$ is  strongly torsionfree in the sense that $M[s]=0$ for \emph{all} $s\in R^0$, then $M$ is pp-simple. For, every proper pp subgroup of $M$ must be bounded, i.e., contained in one of those $M[s]$.
 \end{enumerate}
\end{rem}

Using an idea from \cite[Cor.\,1.7]{HP}, we have:

\begin{prop} Suppose $R$ is a reduced ring and  $M$ is a faithful absolutely pure $R$-module with a maximal (proper) pp subgroup $\phi(M)$.
 \begin{enumerate}[\rm (1)]
 \item 
Then there is\/ $r\in R^0$ such that $\phi(M)=M[r]$, $rM$ is a minimal (nonzero) pp subgroup of $M$,  $rM \cap M[r] = 0$ and $rM + M[r] = M$. Furthermore,

\item if $r$  is not a  right zero divisor, $rM=M$, $\phi(M)=0$, $M$ is pp-simple and  $M[s]=0$ for \emph{all} $s\in R^0$,

\item if $r$  is a right zero divisor,   then $r$ is a right zero divisor, $rM \not= M$ and $M[r]\not=0$ (hence the lattice of pp subgroups of $M$ is not a chain).
\end{enumerate}
\end{prop}
\begin{proof}
As $\phi(M)\not= M$, $\phi$ is not high, hence bounded by some $r\in R^0$, i.e., $\phi(M)\subseteq M[r]$. As $M$ is faithful, maximality implies $\phi(M) = M[r]$. Consider the formula $r^{-1}\phi$ (which, recall, is $\phi(rx)$). Then $r\phi(M)=0$ implies $\phi(M)\subseteq (r^{-1}\phi)(M)$. We claim that also here we have equality. For, otherwise, maximality would yield $(r^{-1}\phi)(M)=M$, i.e., $rM\subseteq \phi(M)$, hence $r^2M=0$ and so $r^2=0$ by faithfulness, which, in turn, would contradict the assumption on the ring.
 
 We now have that $\phi$, $r^{-1}\phi$, and $rx\eq 0$ define the same subgroup, $M[r]$, of $M$. To see that  $rM \cap M[r] = 0$, pick $ra\in  M[r]=\phi(M)$; then $a\in r^{-1}\phi(M)=\phi(M)=M[r]$, hence $ra=0$ as desired.
 Therefore, as $M[r]$ is a maximal proper pp subgroup and $rM\not=0$ by hypothesis, the pp subgroup $rM + M[r]$ must  be equal to $M$. By assumption, $(x\eq x)/\phi$ is a minimal pair in $M$, hence, by modularity of the pp lattice, also $r|x/(x\eq 0)$ is minimal, which means that $rM$ is a minimal pp subgroup $M$. By faithfulness, $rM\not= 0$, this concluding the proof of (1).
 
  (2). Absolutely pure modules are divisible, so if $r$ is not a  right zero divisor, i.e., $\mtx l(r)=0$, then  $rM = M$ (by the definition of divisible given at the outset), hence $\phi(M)  = 0$ by modularity of  $\Lambda$, so $M$ is pp-simple. In that case, finally, given any $s\in R^0$, the subgroup $M[s]$ must be $0$, for otherwise it would have to be all of $M$, which,  in turn, would contradict faithfulness.
  
  (3). Now suppose  $sr=0$ with $s\not= 0$. As $rM$ is not zero, but contained in $M[s]$, we have  $rM\not= M$ by faithfulness. Modularity of $\Lambda$ (together with (1)) shows that $M[r]\not= 0$.
\end{proof}

\begin{cor}\label{abspurered}
If $M$ is a faithful absolutely pure module over a reduced ring $R$, then one of the following holds.
 
  \begin{enumerate}[\rm (1)]
  \item $M$ has an infinite ascending chain of pp subgroups above any given proper pp subgroup.
  \item $M$ is pp-simple (i.e., has no proper pp subgroups other than $0$).
  \item The lattice of pp subgroups of $M$ contains a minimal (nonzero) element and a maximal (proper) element which are incomparable and whose sum is $1$ (i.e., $x\eq x$).  \qed
\end{enumerate} 
\end{cor}

Over a domain, being divisible, absolutely pure modules are faithful. So this special case of the proposition reads as follows.

\begin{cor}\label{abspuredom}
 Suppose $M$ is an absolutely pure module over a domain. 
 
 If $M$  has a maximal proper pp subgroup, $M$ is pp-simple and  torsionfree. Conversely, if $M$ is torsionfree, it is pp-simple (hence has a unique maximal proper pp subgroup). \qed
\end{cor}

\begin{cor}[of the proof]\label{divOredom}
The same holds true for divisible modules over two-sided Ore domains.
\end{cor}
\begin{proof}
  Inspection of the proposition's proof shows that absolute purity proper is used only once, in the very first line, when applying the dichotomy to see that a proper pp subgroup must be bounded. Over a ($\flat\sharp$) domain, a formula is either bounded or cobounded, Prop.\,\ref{flatstardom} below. As in a divisible module, cobounded formulas define everything, only a bounded formula can define a proper subgroup. 
  
  In the remainder of the proof of the proposition absolute purity figures only once more, in {\rm (2)}, where divisibility suffices, as mentioned. Thus the result holds for divisible modules over  ($\flat\sharp$) domains. It remains to note that, by Thm.\,\ref{Ore} below, the ($\flat\sharp$) domains are precisely the two-sided Ore domains.
\end{proof}
 
 For the special case of commutative domains, this is 
  \cite[Cor.1.7]{HP}.

\section{Domains}\label{dom}
In this section we show, among other things, that domains are characterized by the following picture, where the east region of $\Lambda$ is empty, i.e., by the feature that  there be no formula that is bounded and cobounded (see Rem.\,\ref{reformulation} for other ways of expressing that). 

\vspace{3em}
 
 \setlength{\unitlength}{2cm}
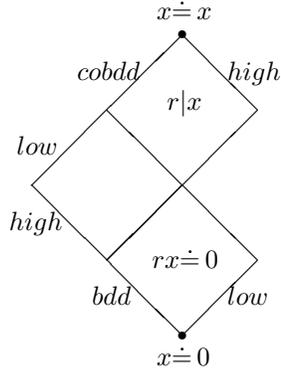
\begin{picture}(2,2)
  \put(0,1){\line(1,1){1}}
 \put(0,1){\line(1,-1){1}}
  \put(1,1){\line(-1,-1){.5}}
  \put(1,1){\line(-1,1){.5}}
   \put(1,1){\line(1,-1){.5}}
    \put(1,1){\line(1,1){.5}}
     \put(1,1){\line(-1,-1){.5}}
      \put(1,0){\line(1,1){.5}}
      \put(1,2){\line(1,-1){.5}}
      
       \put(1.3,1.7){$high$}
         \put(.3,1.7){$cobdd$}
              \put(-.15,.7){$high$}
                 \put(.4,0.2){$bdd$}
                  \put(1.3,0.2){$low$}
               \put(-.1,1.2){$low$}
                  \put(.9,1.5){$r|x$}
                   \put(.8,.45){$rx\dot= \,0$}
                   
                   \put(1,2){\circle*{.05}}
                   \put(1,0){\circle*{.05}}
                   \put(.83,-.2){$x\dot=\, 0$}
                   \put(.83,2.1){$x\dot=\, x$}
 \end{picture}
 \vspace{2em}

%
\begin{rem}
 If $r, s \in R\setminus 0$ and $rs=0$, then $s|x$ is a cobounded formula which is also bounded (by $rx\eq 0$). So, if there are no cobounded bounded formulas, the ring must be a domain. (This also follows from {\rm Cor.\,\ref{highdiv}(5)} and {\rm (8)}.)
\end{rem}

We are going to prove the converse next.
By definition, every nontrivial annihilation formula is bounded. But the dual, that every nontrivial divisibility formula be high, characterizes domains. 

\begin{prop}\label{dich} The following conditions are equivalent for any  ring $R$.
\begin{enumerate}[\rm (i)]
 \item $R$ is a domain (i.e., has no zero divisors).
 \item Every bounded formula (on both sides)  is low.
 \item Every cobounded formula (on both sides)   is high.
 \item Every nontrivial divisibility formula  (on both sides)  is high.
\end{enumerate}
 \end{prop}
\begin{proof} (i) $\Rightarrow$ (ii). If $\phi$ is bounded, we have in particular $r\phi(_RR)=0$ for some $r\not= 0$. As $R$ is a domain, $\phi(_RR)=0$, hence $\phi(F)=\phi(_RR)F=0$ for every flat $R$-module $F$.

\noindent
(ii) $\Rightarrow$ (iii). Let $\phi$ be cobounded. Then its dual is bounded, hence low by (ii), whence $\phi$ 
is high by the Duality Lemma \ref{DualityLemma}.


\noindent
(iii) $\Rightarrow$ (ii).  By (iii), the dual of a bounded formula $\phi$ is high, so $\phi$ is low by the Duality Lemma  again.

\noindent
(ii) $\Rightarrow$ (i). It suffices to show that no nonzero element (of $R$) is a left zero divisor. To this end, let $r\not= 0$. Consider the formula $\phi = (rx\eq 0)$.  Since it is bounded, $\phi$ is low by (ii), hence $\phi(_RR)=0$, and so $r$ is not a left zero divisor.

\noindent
(i) $\Rightarrow$ (iv).  Every injective is divisible, so, over a domain, every nontrivial divisibility formula is high.

\noindent
(iv) $\Rightarrow$ (i). It suffices to show that no nonzero element is a right zero divisor. To this end, let $s\not= 0$ and consider the formula $s|x$. Since it is high by hypothesis, we have $sE=E$ in the 
 injective hull $E$ of $_RR$. Then $1=se$ for some $e\in E$ and therefore $s$ cannot be a right zero divisor. \end{proof}

Invoking the dichotomies  again, we obtain  said converse.

\begin{cor}\label{cordich}
\begin{enumerate}[\rm (1)]
 \item $R$ is a domain if and only if every unary pp formula  is either high or low (and possibly both, see below).
 \item  $R$ is \texttt{not} a domain if and only if there is a unary pp formula that is both bounded and cobounded.
\end{enumerate}
 \end{cor}

 \subsection{$(\flat\sharp)$ domains} Putting the figures at the outset of \S\S \ref{four} and \ref{dom} together, we see that the unary pp lattice of a $(\flat\sharp)$ domain has the shape below, where it does not matter whether is is left or right $(\flat\sharp)$. Neither does it for the few results to come, which is why we stick to the term $(\flat\sharp)$ domain. Besides, we are going to see that, for domains $(\flat\sharp)$ is left-right symmetric.
 \vspace{3em}
 
 \setlength{\unitlength}{2cm}
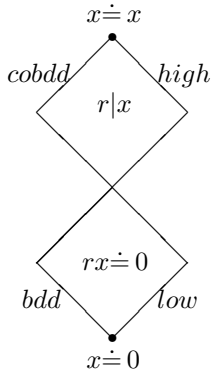
\begin{picture}(2,2)
  \put(1,1){\line(-1,-1){.5}}
  \put(1,1){\line(-1,1){.5}}
   \put(1,1){\line(1,-1){.5}}
    \put(1,1){\line(1,1){.5}}
     \put(1,1){\line(-1,-1){.5}}
      \put(1,0){\line(1,1){.5}}
            \put(1,0){\line(-1,1){.5}}
      \put(1,2){\line(1,-1){.5}}
       \put(1,2){\line(-1,-1){.5}}
      
       \put(1.3,1.7){$high$}
         \put(.3,1.7){$cobdd$}
                 \put(.4,0.2){$bdd$}
                  \put(1.3,0.2){$low$}
                  \put(.9,1.5){$r|x$}
                   \put(.8,.45){$rx\dot= \,0$}
                   
                   \put(1,2){\circle*{.05}}
                   \put(1,0){\circle*{.05}}
                   \put(.83,-.2){$x\dot=\, 0$}
                   \put(.83,2.1){$x\dot=\, x$}
 \end{picture} 
 In words:
 \vspace{2em}

 \begin{prop}\label{flatstardom} Over a $(\flat\sharp)$ domain, the high formulas are exactly the cobounded formulas and the low formulas are exactly the bounded ones, and there is no formula that is both.
 
Consequently, every unary pp formula is either bounded or  cobounded (and not both). \end{prop}
 
 For the special case of commutative domains---which do satisfy $(\flat\sharp)$, see below---the second statement was obtained, in different terminology, by direct calculations in the fraction field in \cite[Predl.1.5]{HP}, and later generalized to two-sided Ore domains in  \cite[Prop.3 and p.253, after Prop.4]{H2}, cf.\ \cite[Fact 2.3 and after]{PP} and \cite[Prop.8.2.1]{P2}. See also \cite[Thm.8.2.28]{P2}.

\begin{cor}\label{highfilterFlatSharp} Suppose $R$ is  a  $(\flat\sharp)$ domain.
\begin{enumerate}[\rm (1)]
\item The set  of all high formulas and all negations of low formulas is a (consistent and) complete type. (This is the type $p_{hi/lo}$ from\/ {\rm Rem.\,\ref{types})}.
 \item The filter of high formulas is generated by the nontrivial divisibility formulas. 
 \end{enumerate}
\end{cor}

\begin{rem}\label{firstUlm}
{\rm (2)} implies that the first Ulm subgroup of an abelian group is the intersection of all high pp subgroups. 
 \end{rem}
 
 In the next section we adopt this as the definition of first Ulm subgroup for modules over arbitrary rings.

\begin{lem}\label{vandeWater} 

\begin{enumerate}[\rm (1)]
\item \emph{(van de Water \cite{Wat})}. If $R$ is left Ore domain, torsionfree divisible left $R$-modules are injective.
 \item If $R$ is a $(\flat\sharp)$ domain, torsionfree divisible $R$-modules are flat.
\end{enumerate}
 \end{lem}

\begin{proof}
For (1), see the proof of \cite[Thm.1]{Wat}. For (2), suppose, $N$ is torsionfree divisible. To show it is flat, we are going to verify that for every unary pp formula $\phi$, we have  $\phi(N)=\phi(_RR) N$ (note, the inclusion from right to left is always true).

If $\phi(_RR)\not= 0$, this follows from divisibility: let $0\not= r\in  \phi(_RR)$, and note that $\phi(N)\subseteq N \subseteq rN \subseteq \phi(_RR) N$.

Assume now, $\phi(_RR)= 0$, i.e., $\phi$ is low.
By Rem.\,\ref{flatsharp}(1), $\phi$ is not high, so it must be bounded, i.e., annihilated by a scalar. As $N$ and $_RR$ ar both  torsionfree, this formula then defines the trivial subgroup $0$ in both of them. In particular, $\phi(N)=\phi(_RR) N$, as desired. By symmetry, the same holds on the right.
\end{proof}

 \subsection{Ore domains}
 It is well known that a left Ore domain $R$ has a left division ring of fractions, $Q$, and that $_RQ$ is an injective envelope of $_RR$. So $_RQ = \E(_RR)$ is a torsionfree and absolutely pure (even injective) $R$-module. (It is always flat on the \emph{other} side, \cite[Exercise 10.20]{LMR} or \cite[Thm.3.6]{G}.)
 
Cor.\,\ref{highdiv}(14) \& (15) show  that 
 a domain $R$ is right  Ore if and only if, for all nonzero ring elements $r$ and $s$, the (left) formula $r\,|\,sx$  is not low. Similarly on the left.

\begin{fact} 
  \cite[Lemma 5.2]{PPR}. The following are equivalent for any domain $R$.
	\begin{enumerate}[\rm (i)] 
	\item $R$ is a left Ore domain.
	\item $\E(_RR)$ is torsionfree.
	\item There is a nonzero torsionfree and absolutely pure left $R$-module.
	\end{enumerate}
 \end{fact}

 \begin{rem} 	 \cite[Lemma 5.2]{PPR} did not state {\rm (ii)}, which follows easily when realizing that the left division ring of fractions \emph{is} $\E(_RR)$. For a direct proof of this, assume, $\E(_RR)$ had torsion, i.e., an element $0\not= a$ with $ra=0$ for some $r\in R^0$. By essentiality, there is $s\in R$ with $0 \not= sa \in R$. By the left Ore condition, $s's = r'r \not= 0$ for some $s', r' \in R$. Thus the assumption would yield zero-divisors: $s'(sa) = r' (ra) = 0$, contradiction.
 \end{rem}

Next we characterize  $(\flat\sharp)$ domains as the two-sided Ore domains, and also as the domains with no high-low formulas. Recall that the last two conditions below are left-right symmetric (and equivalent). 

\begin{thm}\label{Ore}  The following are equivalent for any domain $R$.\stepcounter{enumi}
	\begin{enumerate}[\rm (i)]
	\item $R$ is a two-sided Ore domain.
	\item $\E(_RR)$ is flat, hence $(\flat\sharp)$.
	\item $\E(R_R)$ is flat, hence $(\flat\sharp)$.
	\item $R$ is one-sided $(\flat\sharp)$.
	\item $R$ is two-sided $(\flat\sharp)$.
	\item There is no (left or right) high-low formula.
	\item The unary pp lattices are of the shape shown at the outset of this section.
	\end{enumerate}

 \end{thm}
 
\begin{proof} 
 
 (i)$\implies$(ii) \& (iii)$\implies$(v). Rem.\,\ref{10.17}  says that the right division ring of fractions (whose underlying right $R$-module is $\E(R_R)$) over a right Ore domain is flat as a left $R$-module.  Over a two-sided Ore domain, both injective hulls of $_RR$ and of $R_R$ can be identified and carry the same division ring structure $D$. Consequently, $D$ is flat on both sides, i.e., (ii) \& (iii), hence also (v).
 
 
 Further,  (iv)$\implies$(vi) by Rem. \ref{flatsharp}(1).
(v)$\implies$(iv) and (vi)$\iff$(vii) being trivial,  
  it suffices to verify (vi)$\implies$(i). 
  
 Since by (vi), no formula  is high-low, Cor.\,\ref{RDhi-lo}(4) implies that $R$ is a right Ore domain. As (vi) is symmetric, we get the same on the left.
\end{proof}

\begin{rem}
Here is another proof of\/   {\rm (iv) $\Rightarrow$ (i)}. Assume $R$ is left $(\flat\sharp)$ with a flat and absolutely pure nonzero left $R$-module  $F$. Then, for every pp formula, $\phi(F)=\phi(_RR)F= \ann_F\D\phi(R_R)$ by properties of flat and absolutely pure modules. If $\phi$ is of the form $r|sx$ with $r, s \not=0$, then $\D\phi= \exists(x\eq zs \wedge zr\eq 0)$ and, as $R$ is a domain, $\D\phi(R_R) = 0$. But then $\ann_F\D\phi(R_R)=F$, hence $F=\phi(_RR)F$, and so, as $F\not= 0$, no such formula can be low. By\/ {\rm Cor.\,\ref{highdiv}(14) \& (15)}, $R$ is right Ore. That $R$ is also left Ore follows from the Fact above, since every flat module is  torsionfree.
\end{rem}

Clearly the theorem applies to commutative domains.

\section{Ulm}\label{Ulm}
\subsection{Ulm functors}

We define the Ulm functor $\ulm$  as that which sends a module  to the intersection of all of its high subgroups. (Its functoriality follows from that of pp formulas.)

\begin{rem}
Let $r\in R$. If  $\phi$ is a high formula, then so is $r^{-1}\phi$. This shows that the intersection of all high subgroups of a module is in fact a submodule. 
\end{rem}

This justifies  the following terminology and allows us, as in Ulm theory for abelian groups, to iterate the Ulm functor---for it is, in general, not idempotent.

\begin{definition} \cite{MR???}
 \begin{enumerate}[\rm (1)] 
 \item The  \texttt{Ulm functor} is the functor $\ulm: \RMod \to \RMod$ that sends a module to the  intersection of all of its high subgroups. 
\item  Starting from $\ulm^1 = \ulm$ (or from $\ulm^0 = 1$), iterate $\ulm$ by transfinite induction, setting  $\ulm^{\alpha+1} = \ulm\ulm^{\alpha}$ for every ordinal $\alpha$ and $\ulm^\delta = \bigcap_{\alpha<\delta}\ulm^\alpha$ for every limit ordinal $\delta$.
\item Evaluating these at a module $M$, one obtains its  \texttt{Ulm sequence}\/ $\ulm(M) = \ulm^1(M) \supseteq \ldots \supseteq\ulm^\alpha(M) \supseteq \ldots $, whose first term is called the  \texttt{first Ulm submodule} of  $M$.
\item The smallest ordinal $\tau$ such that $\ulm^{\tau+1}(M) = \ulm^{\tau}(M)$ is called the\/ \texttt{Ulm length} of $M$. Write $\infulm(M)$ for this last term of the Ulm sequence and refer to it as  \texttt{the least (or the last) Ulm submodule} of $M$.

 \item The functor that sends every module to its last Ulm submodule is denoted $\infulm$.
 \item A module $M$ is called \texttt{red\,$\ulm$ced} (or\/ $\ulm$-\texttt{reduced}) if $\ulm(M)=0$, and  \texttt{weakly red\,$\ulm$ced} (or \texttt{weakly} $\ulm$-\texttt{reduced}) if $\infulm(M)=0$. 
 \end{enumerate}
\end{definition}

\begin{rem}\label{UlmFlatSharp}{\texttt{}}

 \begin{enumerate}[\rm (1)] 
 \item Absolutely pure modules have Ulm length $0$.
 \item As pp formulas are existential, $N\subseteq M\implies\ulm(N)\subseteq\ulm(M)$.
 \item Consequently, $\ulm(M)$ contains every Ulm length 0 submodule of $M$.
 \item In particular, red\,$\ulm$ced modules contain no Ulm length 0 submodule (other than $0$).
 \item Every module has a well-defined Ulm length for plain cardinality reasons. But it may differ from module to module, as we shall see, which is why the functor $\infulm$ cannot simply be defined as the intersection of the functors $\ulm^{\alpha}$.
 \item $\ulm^\alpha$ commutes with direct sum and direct product, in particular, it is a p-functor in the sense of Zimmermann \cite{Zim}, for all $\alpha$ (an ordinal or $\infty$).
 \item {\rm Cor.\,\ref{highdiv}(6)} shows that always $$\ulm(M) \subseteq \bigcap_{r\in\Sl} rM.$$
\item By {\rm Cor.\,\ref{highfilterFlatSharp}(2)}, for any module $M$ over a  $(\flat\sharp)$ domain $R$,
 $$\ulm(M) = \bigcap_{r\in R^0} rM.$$
\item In case of abelian groups, the here defined Ulm sequence is therefore  the classical Ulm sequence  and so is length and related terminology of abelian group theory.
 \item Beware: a red\,$\ulm$ced abelian group, that is,  a group with first Ulm submodule $0$, is reduced in the usual sense (i.e., contains no non-trivial divisible subgroup), cf.\ {\rm Rem.\ref{firstUlm}}, but not conversely. Moreover, it is easy to see that the Ulm lengths of an abelian group  and  its  reduced `part' coincide.
 \end{enumerate}
\end{rem}

Remark (7) above suggests the next definition and question.

\begin{definition}
 Let $\ulmd : \RMod\to\Ab$ be the functor that sends a module $M$ to $\bigcap_{r\in\Sl} rM$, the intersection of all its subgroups definable by high divisibility formulas (`high' multiples of $M$).
\end{definition}

\begin{quest}
Over what rings is $\ulm = \ulmd$? Over what RD rings is $\ulm = \ulmd$?
\end{quest}

Remark (8) above states that this is the case for $(\flat\sharp)$ domains. Part of the question is to find out over which rings $\ulmd : \RMod\to\RMod$.
 A common trick using the left Ore condition (usually applied over domains) shows that this is the case over right uniform rings, cf.\,\cite[Lemma 8.2.10]{P2}:

\begin{rem} Let $M$ be a (left) module over  a right uniform ring.

Then $\ulmd(M)$ is a submodule of $M$.
\end{rem}

Next we isolate a condition equivalent to every high formula lying above a high divisibility formula. The condition $\ulm = \ulmd$ seems weaker: it is only needed that the filter generated by the  high formulas is the same as the filter generated  by the high divisibility formulas, which is to say that every high formula lies above a \emph{finite conjunction} of high divisibility formulas. Recall from Rem\,\ref{types}(3) that every (finite) conjunction of high divisibility formulas lies above a single one if and only if the ring is right uniform.

\begin{lem}The following are equivalent for any ring $R$.
\begin{enumerate}[\rm (1)]
 \item  For all\/  $r, s\in R$, if\/ $\lann(r)\subseteq \lann(s)$, then $s\Sl\cap\, rR\not= \emptyset$.
 \item Every high basic RD formula in $_R\Lambda$ lies above a high divisibility formula $t|x$.
\end{enumerate}
\end{lem}
\begin{proof}
By Cor.\,\ref{highdiv}(6),  Rem.\,\ref{T}(1), and Cor.\,\ref{highdiv}(2).\end{proof}

\begin{lem}\label{cohRD}
 Suppose $R$ is  right coherent and left uniform  (i.e., $_RR$ is uniform), and $\phi\in {_R\Lambda}$ is a high formula.
 
 Then either $t|x\leq\phi$ for some $t\in \Sl$ or  there is $t\in R^0$ such that $t\phi$ is low.
\end{lem}
\begin{proof} 
Let $\phi$ be high. By right coherence, $\phi(_RR)$ is f.g., \cite[Thm.2.3.19]{P2}, say $\phi(_RR)=\sum_{i<n} t_i R$. Note, $t_i|x\leq\phi$, since $a=t_ib\in\phi(_RR)M\subseteq \phi(M)$ for all $M$ and all $b\in M$. Hence, if $t_i\in\Sl$, then the first alternative takes place with $t=t_i$. Assume now that none of the $t_i$ are in $\Sl$. Then  $\lann(t_i)\not= 0$ $(i<n)$ and, by left uniformity, there is some $t\in R^0$ in all of them. Thus, $(t\phi)(_RR)=t\phi(_RR)=0$, as desired.
\end{proof}

\begin{rem}
 Let $R$ be a ring with no high-low formulas.
 
 If we could always choose $t\in\Sl$ also in the second alternative of the previous statement, we would have that $\phi$ is bounded, hence not high, which would show that (the high) $\phi$ does lie above some $t|x$ with  $t\in \Sl$. Namely, if $t\phi$ is low, it must be bounded, by hypothesis, whence we get $t'\in R$ with $t't\phi\sim x\eq 0$. As $t\in\Sl$, we obtain $t't\in R^0$, hence $\phi$ bounded.
\end{rem}

Together with Cor.\,\ref{uniformRD}, all this may lead one to conjecture that the following question have an affirmative answer.

\begin{quest}
 Is over  all right coherent and left uniform RD rings $\ulm = \ulmd$?
\end{quest}

\begin{lem}\textbf{}
 \begin{enumerate}[\rm (1)]
 \item Suppose $\ulm(M)=\ulmd(M)$.   
 Then $M/\ulm(M)$ is red\/$\mtx u$ced, i.e., 
$$\ulm(M/\ulm(M)) =0.$$
\item Over general rings we have the same whenever\/ $\ulm(M)$ is pure in $M$.
\item  If $\ulm = \ulmd$ (in particular, if $R$ is a ($\flat\sharp$) domain), then all Ulm factors are red\/$\mtx u$ced.
\end{enumerate}
\end{lem}
\begin{proof} Set $U:=\ulm(M)$.

 (1) (and (3)).  If $a+U\in \ulm(M/U)\subseteq \ulmd(M/U)$ (by Rem.\,\ref{UlmFlatSharp}(7) above), then for all $r\in\Sl$ there is $b\in M$ such that $a-rb\in U\subseteq rM$, hence $a\in rM$ for all such $r$, i.e., $a\in U$ by hypothesis. Thus $a+U =0$, as desired.
 
 (2). If $a+U\in \ulm(M/U)$, then $a+U\in \gamma(M/U)$ for every high $\gamma$. As $U$ is pure in $M$, the coset $a+U$ has a preimage in $\gamma(M)$, hence w.l.o.g., $a\in\gamma(M)$. As $\gamma$ varies, we see that $a\in U$ and hence $a+U =0$ again.
\end{proof}


Much of this entire section is motivated by the search for conditions under which the Ulm sequence stops after the first step, that is, by the question what modules may have Ulm length $\leq 1$. It is not hard to see that this is equivalent to the quest for witnesses inside the first Ulm submodule for the truth of high formulas. To give but a simple example, every element in $U =\bigcap_{r\in R^0} rM$ is divisible by every such ring element $r$, but the witness for that divisibility (the divisor) may not live inside $U$. If it always does, $U$ is divisible in its own right, i.e.,  $U =\bigcap_{r\in R^0} rU$.

Before we  get there, we consider the case where  the Ulm sequence stops before it even starts.

\subsection{Ulm length $0$} Absolutely pure modules clearly have Ulm length $0$. But they are not the only ones. 

We collect some simple observations, mostly based on the fact that, just like any pp formula, high formulas are preserved by homomorphism.

\begin{rem}\label{Ulm0}
\begin{enumerate}[\rm (1)]
\item  The class of Ulm length $0$ modules is closed under direct product and sum and under epimorphic images.
\item As $1\in R$ can be sent anywhere, over a left absolutely pure ring, all modules have Ulm length $0$.
 \item Ulm length $0$ is axiomatized by the closure of all pp pairs $x\eq x/\gamma$ where $\gamma$ is a high formula, so the Ulm length $0$ modules form a definable subcategory (whose dual is, note, the torsionfree class for injective torsion from \cite{MR}, see \cite{MR???}).
 \item  Every module elementarily equivalent to an Ulm length $0$ module has Ulm length $0$.
 \item Over indiscrete rings, where all modules are elementarily  equivalent, \cite{PRZ2}, all modules have Ulm length $0$.
 \item There are indiscrete rings that are not (v.N.) regular (see  \cite{PRZ2} or \cite{P2}). Over such rings there are (Ulm length $0$) modules that are not absolutely pure.
 \item Not every Ulm length $0$ module is elementarily equivalent to an absolutely pure module.
 
 Namely, over left noetherian rings, any module elementarily equivalent to an injective (=absolutely pure) module is injective,  \cite[Thm.3.19 (and Prop.18)]{ES}. Consider a QF 
  ring that is not semisimple. By self-injectivity, all modules have  Ulm length $0$ modules, yet not all are (elementarily equivalent to an) injective.
   \item Over domains, every Ulm length $0$ module is divisible (as nontrivial divisibility formulas $r|x$ are high).
   \item 
Divisible modules over RD rings are absolutely pure,  \cite[Prop.2.4.16]{P2}. 
 Thus {\rm (8)} yields:
\item  
  Ulm length $0$ modules over RD domains are absolutely pure, and
  
\item   pure injective Ulm length $0$ modules over RD domains are injective.
\end{enumerate}
 \end{rem}

\begin{quest}
Over what other rings are pure injectives of Ulm length $0$ injective? 
\end{quest}

Conclude  with the converse to (2) above.

\begin{prop}\label{u=1} The following are equivalent for any ring $R$.
 
\begin{enumerate}[\rm (i)]
\item All left $R$-modules have Ulm length $0$.
 \item $_RR$ has Ulm length $0$.
 \item $\ulm = 1$ (in {\rm $\RMod$}), i.e., all left high formulas are equivalent to $x\eq x$.
 \item  All right low  formulas are equivalent to $x\eq 0$ (this is equivalent to the torsion radical $\tor$  from \cite{MR???} being $0$ (in {\rm $\ModR$}, see  \cite{MR???}).
 \item  $R$ is left absolutely pure.
\end{enumerate}
\end{prop}
\begin{proof} (v)$\implies$(i)$\iff$(ii) has been mentioned.
(ii)$\iff$(iii) is trivial, (iii)$\iff$(iv) is immediate from the duality of high and low, and (iv)$\iff$(v) is \cite[Prop.\,2.32]{MR} (a proof using pp formulas is  given in \cite{MR???}).
\end{proof}

\subsection{Finding witnesses}
 Recall  that a \texttt{$\wedge$-atomic} formula is a finite conjunction of atomic formulas, which, in modules, means a finite system of linear equations over the ring.
 We often split the free variables of a formula into two kinds, $\br x$ and $\br y$. This is to indicate which are to stay free later, $\br x$, and which later may get quantified out, $\br y$. This is  just a notational device. For instance, we may write a $\wedge$-atomic formula $\psi$ as $\psi(\br x, \br y)$ if we intend to consider the pp formula $\phi=\phi(\br x) = \exists \br y \psi(\br x, \br y)$ later.

\begin{definition}[Notation] \label{subgamma}
Let $\gamma$ be a unary pp formula.
 
%
%
 \begin{enumerate}[\rm (1)] 
\item  For an arbitrary tuple of variables, $\br x = (x_0, \ldots, x_{n-1})$,  the formula $\bigwedge_{i<n} \gamma(x_i)$ is denoted by $\br \gamma(\br x)$.

  \item For any pp formula $\psi = \psi(\br x)$, let   $\psi_\gamma$ denote the pp formula $\psi(\br x) \wedge \br \gamma(\br x)$. 
  
 \item  If  $\phi=\phi(\br x)= \exists \br y \psi(\br x, \br y)$ is an existential formula with $\psi(\br x, \br y)$ quanifierfree, $\phi^\gamma$ denotes the formula $\exists \br y \psi_\gamma(\br x, \br y)$, that is, the formula $\exists \br y \, (\psi(\br x, \br y) \wedge \br \gamma(\br x) \wedge \br \gamma(\br y))$.
 
 \item For $\phi$ as in {\rm (3)}, set $\Gamma_\phi = \Gamma_\phi(\br x, \br y) = \{ \psi_\gamma(\br x, \br y) : \gamma\in \Gamma
  \}$, where  
 \item $\Gamma$ denotes the set of high formulas.
  \end{enumerate}
\end{definition}

Note that in (2) and (3), $(\cdot)_\gamma$ and $(\cdot)^\gamma$ do not change the collection of free variables in a formula, they simply make (the witnesses for) $\br y$ satisfy $\gamma$.

\begin{rem}\label{subgammarem} Let $\gamma$ and $\delta$ be unary pp formulas and  $\phi=\phi(\br x)$ and $\chi=\chi(\br x)$  pp formulas in $\br x$.

\begin{enumerate}[\rm (1)]
\item  {\rm (a)}\hspace{0.5em} $\phi_\gamma\wedge\chi_\gamma \sim  (\phi\wedge\chi)_\gamma $\hspace{1em} and\hspace{1em} {\rm (b)}\hspace{0.5em} $\phi_\gamma\wedge\phi_\delta \sim  \phi_{\gamma\wedge\delta}$.

\item 
 {\rm (a)}\hspace{0.5em} $\phi^\gamma\wedge\chi^\gamma  \sim (\phi\wedge\chi)^\gamma$,\hspace{1em}  but in general only\hspace{1em} {\rm (b)}\hspace{0.5em}  $\phi^{\gamma\wedge\delta}\leq \phi^\gamma\wedge\phi^\delta$.


 \item  If\/ $\gamma$ and $\delta$ are high, so is $\delta^\gamma$. 
 
\item In particular, if $a\in \ulm(M)$ and $\phi= \exists \br y \psi(x, \br y)$ is high, then, for every high formula $\gamma$, there is a witness  in $\gamma(M)$ for the truth of $\phi(a)$ in $M$, i.e., there is an $l(\br y)$-tuple $\br b$ in $\gamma(M)$ such that $a\in \psi(M, \br b)$.

\item Given a module $M$ and a matching tuple $\br u$ in $\ulm(M)$, any realization in $M$ of the type $\Gamma_\phi(\br u, \br y)$ (of\/ {\rm Def.\,\ref{subgamma}(4)} above, with $\br u$ substituted for $\br x$) is a witness inside $\ulm(M)$ for the truth of $\phi(\br u)$ in $M$, and consequently, for its truth in $U$. 

\item If $\phi$ is  quantifier-free, $\phi^\gamma$ \emph{is} $\phi_\gamma$.

\item 
\emph{Any} consistent existential formula is consistent with every high formula, for the simple reason that if it is true in a module, being existential, it is also true in the module's injective envelope, where \emph{everything} is high.

\item More interestingly, given an existential formula $\psi(\br x, \br y)$ and a matching tuple $\br a$ in $M$, if $\psi(\br a, \br y)$ is consistent with $M$ (i.e., $\psi(\br a, \br y)$ has a solution for $\br y$ in some extension of $M$), then so is $\psi(a, \br y)\wedge \br\gamma(\br y)$. 

\item In particular, suppose $\phi=\phi(\br x)= \exists \br y \,\psi(\br x, \br y)$ is an arbitrary existential formula with $\psi(\br x, \br y)$ quantifier-free and $\br a$ is a matching tuple  in $\gamma(M)$ with $\gamma$ high.  If $\phi(\br a)$ is consistent with $M$, then so is $\phi^\gamma(\br a)$.

\end{enumerate}
\end{rem}


Ideally we'd like to have (4) simultaneously for \emph{every} high formula $\gamma$, so that witnesses end up in the first Ulm submodule. In general, this cannot be done, as may be seen from the existence of abelian groups with  Ulm length $>1$. But it is possible in pure injective modules, as will be shown in Thm.\ref{length1} below (where the obvious vehicle  are types as in (5)). 

One case where we can work in just the free variables of $\phi$ is the following.

\begin{lem}\label{uniquewitness}
 Suppose the  pp formula  $\phi=\phi(\br x)= \exists \br y \psi(\br x, \br y)$ has \texttt{unique witnesses} in a module $M$, by which we mean, that for every $\br a\in \phi(M)$ there is only one $\br b$ in $M$ witnessing the truth of $\phi$, i.e., $\psi(\br a, M) = \{\br b\}$.
 
 If $\phi$ is high (with unique witnesses), then $\ulm(M)\subseteq \phi(\ulm(M))$.
 
\end{lem}
\begin{proof}
 Let $\br b$ be a witness, i.e., $a\in \psi(M, \br b)$. By uniqueness of $\br b$, it suffices to verify  that there is a witness in each high $\gamma(M)$. But this is clear, since $a$ satisfies the (high) formula $\phi^\gamma$. 
\end{proof}

\begin{rem}
The formula $\exists y ( rx \dot= y)$ ($\sim x\eq x$) trivially has uniqueness of witnesses, which shows, once again, that $\ulm(M)$ is a submodule of $M$. A less trivial example is the formula $\exists y ( ry \dot= x)$ (i.e., $r|x$) for  nonzero $r$ and $M$, a torsionfree module over a domain. This leads to the following.
\end{rem}

\begin{prop}  Every torsionfree module over a two-sided Ore (= $(\flat\sharp)$) domain has injective first Ulm submodule and therefore Ulm length $\leq 1$.

Moreover, if $M$ is such a module, then $M = \ulm(M) \oplus M_{r}$, with $M_{r}$ a red\,$\ulm$ced module, and 
 $\ulm(M)=\infulm(M)$, the largest injective submodule of $M$.
\end{prop}
\begin{proof}
By Cor.\ref{highfilterFlatSharp}(2) (and Thm.\,\ref{Ore}), every high formula is above a nontrivial divisibility formula, which clearly has unique  witnesses in a torsionfree module over a domain. By Lemma \ref{uniquewitness}, $\ulm(M)\subseteq r\ulm(M)
$ for every nonzero $r\in R$, hence $\ulm(M)\subseteq \ulm(\ulm(M))$  (Rem.\,\ref{UlmFlatSharp}(7)). Thus $\ulm(M)$ has Ulm length $0$ and is therefore divisible, Rem.\,\ref{Ulm0}(8). But it is also torsionfree, hence injective by Lemma \ref{vandeWater}(1).

As $\ulm(M)$ has Ulm length $0$ and $\ulm$ commutes with $\oplus$, the first Ulm submodule of the complement must be $0$, hence contains no nonzero nontrivial injective, Rem.\,\ref{UlmFlatSharp}(4).
\end{proof}

Note that torsionfreeness was essential in the proof: even for commutative domains, divisible implies injective only (and precisely) if the domain is Dedekind \cite[Cor.3.24, p.73]{LMR}.

\subsection{Pure injective modules}
It is well known that the first Ulm subgroup of a pure-injective abelian group is divisible, hence of Ulm length $0$, \cite[Exercise 6.1(1)]{F}.
    We generalize both the result and the proof to arbitrary rings, expanding an  idea of \cite[\S 3.3]{Mac}. 
   
\begin{thm}\label{length1} 
 The first Ulm submodule of a pure injective module is a pure injective module of Ulm length 0. 
 
Consequently, pure injective modules have Ulm length $\leq 1$.
\end{thm}
\begin{proof} Suppose $N$ is pure injective. Set $U=\ulm(N)$.

To prove $U$ is pure injective, consider an arbitrary $1$-type $p$ of pp formulas with parameters from $U$ that is finitely satisfied in $U$. We have  to show that $p$ is realized in $U$, \cite[Lemma 4.2.1]{P2}.
 
 To this end, set $p^\Gamma = \{\phi^\gamma : \phi\in p, \gamma\in\Gamma\}$. It is not hard to see that it suffices  to show that $p^\Gamma$ is realized in $N$ (as then any realization, as well as the corresponding witnesses, automatically end up in $U$). 
 
 Since $N$ is pure injective, it remains to verify that $p^\Gamma$ is finitely satisfied in $N$. Rem.\ref{subgammarem}(2b) shows that it suffices 
   to show that   every single formula  $\phi^\gamma\in p^\Gamma$ is satisfied in $N$.  
 But $p$ was finitely satisfied \emph{in} $U$, which means that the realization as well as the witnesses for every $\phi\in p$ can be chosen in $U$. In particular, 
  the formula $\phi^\gamma$ is satisfied in $U$ (or equivalently, in $N$),  for every  $\gamma\in\Gamma$, which completes the proof that $U$ is pure injective.
  
To prove $U=\ulm(U)$, i.e., that $U$ has Ulm length $0$, we have to verify that every element $u\in U$ not only satisfies all high formulas in $N$ (which it does, by definition), but that there are witnesses for that truth inside $U$ itself, so that $u$ satisfies all high formulas inside $U$. To this end, let  $\phi= \exists \br y \psi(x, \br y)$ be a high formula with $\psi$ quantifier-free. By Rem.\ref{subgammarem}(5),  it suffices to realize the type $\Gamma_\phi(u, \br y)$ in $N$, for which, by pure injectivity, it suffices to finitely realize it in $N$, hence,  by Rem.\ref{subgammarem}(1) now, to realize every single formula $\psi_\gamma(u, \br y)\in\Gamma_\phi(u, \br y)$  in $N$. This, however, is the same as to say that $u$ satisfies  $\phi^\gamma$ in $N$, which is clear, as $\phi^\gamma$ is high, Rem.\ref{subgammarem}(3). 
\end{proof}


\begin{cor}\label{elemequ} Every module is elementarily equivalent to, in fact an elementary submodule of, a module of Ulm length $\leq 1$. 
\end{cor}

\begin{cor}\label{decRD}
 If $N$ is a  pure injective module over an RD domain, then $\ulm(N)$ is injective and\/ 
  $N = \ulm(N) \oplus N_{r}$, with\/ $N_{r}$ a red\,$\ulm$ced submodule, and\/ 
 $\ulm(N)=\infulm(N)$,  the largest injective submodule of $M$.
\end{cor}
\begin{proof}
 By the theorem, $\ulm(N)$ is pure injective of Ulm length $0$, hence injective by Rem.\,\ref{Ulm0}(11). As $\ulm$ commutes with direct summand, $\ulm(N_{r})$ must be $0$, hence $N_r$ contains no nonzero nontrivial injective, Rem.\,\ref{UlmFlatSharp}(4).
\end{proof}

\begin{quest}
 What can be said if the RD ring is no longer a domain?
\end{quest}

{\small The  theorem  has little to do with modules. Inspection of the argument shows that it is true in the following form for any first order language $L$.

\begin{thm}\label{generalL}
Let $\Gamma\subseteq L_1$ be a set of unary $L$-formulas. Let $\ulm$ be the operator of (infinitary) conjunction of formulas from $\Gamma$, i.e., $\ulm(M) = \bigcap_{\gamma\in\Gamma}\gamma(M)$ for every $L$-structure $M$.
 
  Suppose $\Gamma\subseteq L_1$ has the following closure properties.
  \begin{enumerate}[\rm (a)]
 \item $\Gamma$ is closed under conjunction.
 \item $\Gamma$ is closed under $(\cdot)^\gamma$, i.e., for all $\phi, \gamma \in \Gamma$ we have $\phi^\gamma\in\Gamma$.
 \item $\ulm(M)$ is a substructure of $M$ for every $L$-structure $M$. (This is true if $L$ is purely relational. In modules with $\Gamma$, the high formulas, this is the  case, because, on the one hand, pp formulas define subgroups, and, on the other, high formulas are closed under substituting (the unary operation of) scalar multiplication.)
\end{enumerate}

If $N$ is an algebraically (= atomic) compact (= pure-injective) $L$-structure,  then $U = \ulm(N)$ is algebraically compact and $U = \ulm(U)$.\qed
\end{thm}
}

 \subsection{When things are pure} The fact that $U = \ulm(U)$ is some weak kind of purity of $U$ in $N$, namely, existential closedness of $U$ in $N$ w.r.t.\ high formulas (with witnesses inside $U$). We next look at when $U$ is fully pure in $N$ (i.e., existentially closed w.r.t.\ \emph{all}\/ pp formulas). 


\begin{definition} [Notation]
 Consider the condition

\hspace{2em}   $(\ulm_\phi)$\hspace{3em} $U\cap \phi(N) \subseteq \phi^\gamma(N)$,  for every  $\gamma\in\Gamma$. 
\end{definition}

\begin{rem}\label{finiteconj}
 \begin{enumerate}[\rm (1)] 
 \item The set of unary pp formulas $\phi$ for which  $(\ulm_\phi)$ holds is closed under finite conjunction,  {\rm Rem.\,\ref{subgammarem}(2a)}.
 \item  $(\ulm_\phi)$ holds for every high formula $\phi$. Moreover, $U\subseteq\phi^\gamma(N)$ for every $\gamma\in\Gamma$,  {\rm Rem.\ref{subgammarem}(3).}
 \item  $(\ulm_\phi)$ holds for every quantifier-free unary pp formula $\phi$,  {\rm Rem.\ref{subgammarem}(6)}.
 \item  Consequently, $(\ulm_\phi)$ holds whenever  $\phi$ is a  finite conjunction of high formulas and quantifier-free formulas.
 \end{enumerate}
\end{rem}
 
 \begin{prop}\label{uphi} 
 
 \noindent 
If  $N$ is   pure injective, the following are equivalent.
\begin{enumerate}[\rm (i)] 
\item $U$ is pure in $N$.
  \item $U$ is a direct summand of $N$.
 \item  $(\ulm_\phi)$ holds for every unary pp formula $\phi$.
  \item  $(\ulm_\phi)$ holds for every unary formula $\phi$ not in $\Gamma$. 
\end{enumerate}
\end{prop}
 \begin{proof}
By Thm.\ref{length1}, $U$ is pure injective. Thus the first equivalence follows from the definition of pure injectivity.
For the second, first of all  keep in mind that for purity in modules,  unary pp formulas suffice \cite[Prop.2.1.6]{P2}. 

 It is not hard to see that $U$ is pure in $N$ if and only if, for every unary pp formula $\phi$ and every $u\in \phi(N)$,  the type $\Gamma_\phi(u, \br y)$ of Def.\,\ref{subgamma}(4)  is realized in $N$. But $\Gamma_\phi(u, \br y)$ is realized in the pure injective module $N$ if and only if it is  finitely satisfied in $N$.  As conjunctions of high formulas are high and $\psi_\gamma\wedge \psi_\delta$ is equivalent to $\psi_{\gamma\wedge\delta}$, this is equivalent to every single $\psi(u, \br y)\in \Gamma_\phi$ being satisfiable in $N$, which, in turn, is equivalent to $\ulm_\phi$.  In other words, $U$ is pure in (a pure injective) $N$ iff $\ulm_\phi$ holds for all unary pp formulas $\phi$. 

Finally, it is clear from  (the end of the proof of) Thm.\ref{length1} that 
$(\ulm_\phi)$ holds for all high formulas $\phi$. Hence (iv) suffices, which completes the proof.
\end{proof}

{\small 
\begin{rem}
Similarly to Thm.\ref{length1}, this is true mutatis mutandis (direct summand to be replaced by retract) for arbitrary languages $L$ and sets\/ $\Gamma$ as in {\rm Thm.\,\ref{generalL}}, with one proviso: if purity can't be reduced to unary pp formulas in $L$-structures, one has to allow $\phi$ of \emph{all} arities in the conditions $(\ulm_\phi)$.
\end{rem}
}

\begin{cor}\label{pidecomp} Suppose  $N$ is a  pure injective module. 



Then the following are equivalent.
\begin{enumerate}[\rm (i)] 
\item $\ulm(N)$ is pure in $N$.
  \item $N = \ulm(N) \oplus N_{r}$ with $N_{r}$ a red\,$\ulm$ced pure injective module.
  \item $N = N_0 \oplus N_{r}$ with $N_0$ of Ulm length $0$ and  $N_{r}$ a red\,$\ulm$ced pure injective module.
  \item  $(\ulm_\phi)$ holds for every bounded formula $\phi$.

\end{enumerate}
\end{cor}

\begin{proof} Invoking Dichotomy I, the lemma shows (i)$\iff$(iv) and these imply that $N$ can be written as $N = \ulm(N) \oplus N_{r}$. The rest of (ii) follows as in Cor.\,\ref{decRD}. Thus (i)$\implies$(ii) (whose converse is trivial). By the theorem, (ii)$\implies$(iii).
%
For the converse, let $N = N_0 \oplus N_{r}$ as in (iii). Then 
$\ulm(N) = \ulm(N_0)  \oplus \ulm(N_{r}) = \ulm(N_0)  \oplus 0=\ulm(N_0)=N_0$, which  concludes the proof.
\end{proof}

Prop.\,\ref{uphi} and Rem.\,\ref{finiteconj} yield at once

\begin{cor}\label{finiteconjdec}
 Suppose $R$ is a  ring over which  (all bounded formulas are quantifier-free or, more generally) every unary pp formula is equivalent to a finite conjunction of high formulas and quantifier-free formulas.
 
 Then every pure injective $R$-module $N$  decomposes as
$N_0 \oplus N_{r}$ with $N_0$ a (pure injective) module of Ulm length $0$  and  $N_{r}$ a red\,$\ulm$ced pure injective module.
\end{cor}

\begin{rem} We already know that in pure injective modules over RD domains the first Ulm submodule splits off, see {\rm Cor.\,\ref{decRD}}. The proof was based on {\rm Rem.\,\ref{Ulm0}(9)}, the lucky fact that  it was pure, even absolutely.

 Another proof  is this:  recall from \S\ref{1} that every unary pp formula  is equivalent to a finite conjunction of formulas of the form $r|sx$ where $r, s \in R$. By\/  {\rm Cor.\,\ref{RDhi-lo}(1)}, these are  high or quantifier-free.
\end{rem}

\begin{quest}
 Is there a reasonable class of rings or modules for which some version of Ulm's Theorem holds?
\end{quest}

\begin{quest}
 Over what rings are Ulm submodules of cotorsion modules cotorsion, resp., are their Ulm factors pure injective?
\end{quest}

This is true in abelian groups, see \cite[Prop.8.3, p.291]{F}. (Here cotorsion is \emph{not} in the sense of \cite{MR},\footnote{Sorry, Alex!} but in the sense of Enochs that every submodule with  flat cokernel splits.)

{\footnotesize \subsection{An entertaining exercise}

Of course,  a pure injective pure submodule is a direct summand. But it may be an entertaining exercise to write down a pp type whose realization in $U=\ulm(N)$ directly yields a retract from $N$ to $U$. 

\emph{Solution.} Suppose $N$ is pure injective and $U = \ulm(N)$ . Consider  the pp type of $N\setminus U$ over $U$ in variables $\{x_a : a\in N\setminus U\}$. This type is finitely satisfied in $N$ (and thus realized in the pure injective $N$) when $U$ is pure in $N$.

Here are the details.  Consider $\at(N)$, the atomic type of $N$ in the variables $\{ x_a : a\in N \}$, i.e., the set of all atomic formulas (=homogeneous linear equations) $\alpha(x_{a_1}, \ldots, x_{a_n})$ that are true of $(a_1, \ldots, a_n)$ in $N$.  
So a realization of the type $\at(N)$ is a sequence $\langle b_a : a\in N \rangle$ such that $(b_{a_1}, \ldots, b_{a_n})$ satisfies $\alpha(x_{a_1}, \ldots, x_{a_n})$.  As pure injective = atomically compact = algebraically compact, an atomic type has a realization in $N$, whenever every finite subtype (=subset) does (cf.\   \cite{P1} or \cite{P2} for modules, or \cite{H} for arbitrary $L$-structures).
 
 Form $\at(N/U)$ by replacing, in (formulas from) $\at(N)$, all  entries of $x_u$, where $u\in U$, by $u$ itself and adding all equations of the form $x_u \eq u$ for all $u\in U$. Clearly, $\at(N/U)$ is equivalent to the set of all atomic formulas with parameters from $U$ (=inhomogeneous linear equations over $U$) that are true of their respective elements in $N$. In other words, a realization of $\at(N/U)$ is a realization $\langle b_a : a\in N \rangle$ of $\at(N)$ in which $b_u = u$ for every $u\in U$.
 
 Set  $\Phi := \at(N/U) \cup \bigcup_{a\in N, \gamma\in\Gamma} \gamma(x_a)$. Any realization of $\Phi$ in $N$ lies entirely in $U$. Moreover, as every $u\in U$ is the only element of $N$ satisfying the formula $x_u \eq u \in \at(N/U)$, 
  any such realization exhausts all of $U$. More importantly yet, given any realization $\langle b_a : a\in N \rangle$ of $\Phi$, the assignment $a\to b_a$ defines a retract $h$ from $N$ to $U$, i.e.,  a homomorphism $h: N \to U$ that is the identity on $U$ (that is, rather a retract of the identical inclusion of $U$ into $N$), for $b_u$ must be $u$ for all $u\in U$.
  Note that any such retract is a direct summand, hence a pure submodule and itself a pure injective module.

It remains to verify that every finite subset $\Psi$ of $\Phi$ is satisfied in $N$. Any such $\Psi$ consists of finitely many atomic formulas with parameters from $U$ (including formulas of the form $x_u \eq u$) and high formulas $\gamma_j(x_b)$ for some $b\in N$. Let $\br x$ be the tuple of all free variables and $\br u$ the tuple of all parameters  from $U$ occurring in $\Psi$. Consider $\alpha(\br x, \br u)$, the conjunction of all  formulas in $\at(N/U)\cap \Psi$, and $\psi(\br x; \br u)$, the conjunction of all of $\Psi$ (i.e., including those high $\gamma_j$'s). 

If $\br x = (x_{a_1}, \ldots, x_{a_n})$, the tuple $(a_1, \ldots, a_n)$ satisfies $\alpha(\br x, \br u)$ in $N$, hence $\exists\br x\,\alpha(\br x; \br u)$ is satisfied by $\br u$ in $N$. Purity of $U$ in $N$ implies that it is so in $U$. Consequently, all witnesses satisfy all high formulas and, in particular, also $\psi(\br x; \br u)$  in $N$, as desired.\qed}

\section{Conclusion or: A cobounded functor?}\label{concl}
Consider some 
extreme properties of a ring.

\begin{enumerate}[\rm (a)]
\item All high formulas are equivalent to $x \eq x$.
\item  All low formulas  are equivalent to $x \eq 0$. 
\item All cobounded formulas are equivalent to $x \eq x$.
\item  All bounded formulas  are equivalent to $x \eq 0$.
\end{enumerate}

By the duality lemma, the first on the left is equivalent to the second on the right. The same holds for the third and the fourth.
What can one say about rings with the latter two? More generally, is it worth the while to introduce Ulm's twin brother $\twinulm$, a co-Ulm functor that is based on cobounded formulas in place of high ones?

Clearly, over $(\flat\sharp)$ domains, $\twinulm=\ulm$, simply because high=cobounded. Over general domains, the modules of co-Ulm length $0$  are exactly the divisible modules, while $\Psi_{bdd}$ is simply the classical torsion radical. (Cf.\  \cite{MR???} for a comparison of this with the injective torsion radical, which is  the one given by $\Psi_{low}$.)

The ring extreme, $\twinulm = 1$ (or `all modules have co-Ulm length 0'), is certainly not very interesting. It happens only 
 over division rings. Namely, if everything satisfies every cobounded formula, then everything is divisible, so every ring element is a unit.

\end{document}